\documentclass[journal,web]{ieeecolor}
\usepackage{subcaption}
\usepackage{amsmath}
\usepackage{amssymb}
\usepackage{amsfonts}
\usepackage{generic}
\usepackage{mathrsfs}
\usepackage{color}
\usepackage{cite}
\usepackage{algorithmic}
\usepackage{textcomp}
\usepackage[mathscr]{euscript}
\usepackage{graphicx} 
\usepackage{epstopdf} 
\usepackage{mathptmx} 
\usepackage{times}
\usepackage{float}
\usepackage{array}
\usepackage{booktabs}

\usepackage[colorlinks,
            linkcolor=red,
            anchorcolor=blue,
            citecolor=green
            ]{hyperref}
\newcommand{\dn}{\mathbf{d}}
\newcommand{\R}{\mathbb{R}}
\newcommand{\diag}{\mathrm{diag}}

\newcommand{\TL}{\tilde{\mathcal{L}}}
\newcommand{\tdn}{\tilde{\mathbf{d}}}
\newtheorem{lem}{Lemma}
\newtheorem{remark}{Remark}
\newtheorem{thm}{Theorem}
\newtheorem{Coro}{Corollary}
\newtheorem{Def}{Definition}

\def\BibTeX{{\rm B\kern-.05em{\sc i\kern-.025em b}\kern-.08em
    T\kern-.1667em\lower.7ex\hbox{E}\kern-.125emX}}
\definecolor{protect-eyes}{RGB}{199,237,204}
\begin{document}

\graphicspath{{figure/}}

\title{Finite-time Non-overshooting Leader-following Consensus Control for Multi-Agent Systems}
\author{Min Li, \IEEEmembership{Graduate Student Member, IEEE}, Andrey Polyakov, Siyuan Wang, and Gang Zheng, \IEEEmembership{Senior Member, IEEE}
\thanks{This work was supported in part by the China Scholarship Council (CSC) under Grant 202106160022 (Corresponding author: Siyuan Wang).}
\thanks{Min Li, Andrey Polyakov, and Gang Zheng are with the Inria, Univ. Lille, CNRS, UMR 9189 - CRIStAL, Centrale Lille, F-59000 Lille, France (e-mail: min.li@inria.fr, andrey.polyakov@inria.fr, gang.zheng@inria.fr).}
\thanks{Siyuan Wang is with the State Key Laboratory of Virtual Reality Technology and Systems, Beihang University, Beijing, China (e-mail: siyuanwang0603@buaa.edu.cn).}}

\maketitle

\begin{abstract}
This paper addresses the finite-time non-overshooting leader-following consensus problem for multi-agent systems, whose agents are modeled by a dynamical system topologically equivalent to the integrator chain. Based on the weighted homogeneity, a nonlinear consensus control protocol is designed. A tuning scheme ensures the finite-time stability of the consensus error such that the agents do not have overshoots in the first component of the state vector. Simulations are presented to demonstrate the effectiveness of the proposed design.


\end{abstract}
\label{sec:introduction}
\section{Introduction}\label{Intro}

Consensus in Multi-Agent Systems (MAS) has been an important research topic over the past two decades. This study seeks some control protocol to ensure all agents' states reach an agreement. In particular, the leader-following consensus, where the agreement is the leader's state, has received substantial focus \cite{ni2010leader,guan2012finite,li2023generalized}. Despite considerable research progress, there has been limited research addressing leader-following consensus in the context of both time constraints, which affect the convergence rate, and space constraints, which impact the safety of agent movements.

The primary technique involved in the control problem considering time constraints is the finite-time stability strategy, which ensures the system state reaches equilibrium at a time dependent on the initial system state. This method, known since the 1960s \cite{fuller1960relay,korobov1979solution,haimo1986finite,bhat2000finite}, has seen considerable development and application in recent years, particularly in addressing the consensus problem for the MAS \cite{wang2008finite,guan2012finite,li2023generalized}. 

On the other hand, numerous results have focused on the consensus under space constraints. Specifically, constrained consensus, first introduced in \cite{5404774} and actively studied by \cite{lin2013constrained,liu2017constrained,qiu2016distributed}, aiming at a consensus limiting each agent to a closed convex set. Additionally, research on positive systems \cite{de2001stabilization, rami2007controller, farina2011positive} contributes to the consensus study, leading to the positive consensus, which means a consensus with agent states being positive \cite{6681911,7892854,8862888,su2017positive,su2018positive}. 
Moreover, the study of non-overshooting stabilization \cite{phillips1988conditions,krstic2006nonovershooting,schmid2010unified}, which ensures that a partial system state is uniformly bounded by its terminal value, has been integrated into the consensus research by \cite{8378231}. 
Nevertheless, few works focus on consensus considering both time and space constraints. This topic is vital as it contributes to the MAS's working efficiency and safety. 

This paper aims to address the consensus problem under both time and space constraints. Specifically, we tackle the finite-time non-overshooting leader-following consensus problem for MAS with high-order integrator dynamics. The MAS operates under a directed graph that allows local transmission of both the agents' states and a supporting vector using only local information, similar to the configurations in \cite{ning2020bipartite, hong2006tracking, ren2007multi, liu2011synchronization}. Regarding this MAS, we propose a consensus control protocol that guarantees finite-time convergence, improving upon \cite{li2023generalized} by a distributed design. The main technique for achieving finite-time convergence is homogeneity-based control.

Homogeneity is a property characterizing a wide class of systems, including all linear systems and a large subset of nonlinear systems. An asymptotically stable homogeneous system can exhibit finite-time convergence with a negative homogeneity degree. This motivates the ``upgrading" from linear control to homogeneous control \cite{wang2021generalized, polyakov2020generalized}, which is one of the most effective ways to obtain a homogeneous control.

Based on the above establishment, we propose a linear protocol that ensures asymptotic non-overshooting consensus, where the followers' first state component does not exceed the one of the leader. This non-overshooting behavior is ensured by the strict positive invariance of the consensus error within a linear cone. Building on this, the linear protocol is upgraded to a homogeneous one, resulting in a finite-time non-overshooting consensus protocol. This ensures consensus error achieves finite-time stability and strict positive invariance within a homogeneous cone. Additionally, a class of disturbances affecting agent dynamics is characterized, under which the homogeneous protocol can maintain the non-overshooting property. Finally, simulations confirm the control effectiveness.

The remainder of this paper is structured as follows. Section \ref{PF} describes the problem to be studied. Section \ref{Pre} provides some useful knowledge and results.  Section \ref{MR2} outlines the design of the linear non-overshooting consensus protocol and its upgrade to a homogeneous one. Section \ref{SR} presents simulations. Finally, Section \ref{Con} concludes this work.

\textit{Notations}:
$\R$ is the set of real numbers; 
$\R_+$ ({\it{resp.,}} $\R_{-}$) is the set of positive ({\it{resp.,}} negative) real numbers; 
$\R_{\geq0}\!\!=\!\!\R_+\!\!\cup\!\{0\}$;
$\mathbb{N}_+$ is the set of positive integers;
a series of positive integers $1,\dots,N$ is denoted as $\overline{1,N}$; 
let $n\!\!\in\!\!\mathbb{N}_+$, $\R^n$ and $\R^{n\times n}$ denote the  $n\!\!\times\!\! 1$ real vector and the $n\!\!\times\!\! n$ real matrix, respectively; 
$\R^n_{+}$ ({\it{resp.,}} $\R^n_{-}$) denotes the $n\!\!\times\!\! 1$ real vector whose elements are positive ({\it{resp.,}} negative);
$I_n$ is the $n\!\!\times\!\!n$ identity matrix;  
$\diag\{\sigma_i\}_{i=1}^n$ is the (block) diagonal matrix with the diagonal entry $\sigma_i$ of proper dimension; 
$\{\sigma_{ij}\}$ is the $n_1\!\!\times\!\!n_2$ (block) matrix of a proper dimension with element $\sigma_{ij}$, $i\!\!=\!\!\overline{1,n_1}$, $j\!\!=\!\!\overline{1,n_2}$, $n_1,n_2\!\!\in\!\!\mathbb{N}_+$; 
$\mathbf{1}_{n}\!\!\in\!\!\R^n$ ({\it{resp.,}} $\mathbf{0}_{n}\!\!\in\!\!\R^n$) with elements are all ones ({\it{resp.,}} zeros);  
let $P\!\!=\!\!\{P_{ij}\}\!\!\in\!\!\R^{n_1\!\times\!n_2}$, $P_{ij}\!\!\in\!\!\R$, $P\!\!\geq\!\!0 ({\it{resp.,}} \leq\!\!0)$ means $P_{ij}\!\!\geq\!\!0({\it{resp.,}} \leq\!\!0)$, $\forall i\!\!=\!\!\overline{1,n_1}$, $\forall j\!\!=\!\!\overline{1,n_2}$;
let $P\!\!\in\!\! \R^{n\times n}$, $P\!\!\succ\!\! 0( {\it{resp.,}}\!\!\prec\!\! 0)$ means that $P$ is symmetric and positive ({\it{resp.,}} negative) definite; $P$ is anti-Hurwitz if $-P$ is Hurwitz; 
$\exp(P)\!\!=\!\!\sum_{i=0}^\infty\!\!\tfrac{P^i}{i!}$;
$\lambda_{\max}(P)$ ($resp.,$ $\lambda_{\min}(P)$) represents the maximum ($resp.,$ minimum) eigenvalue of $P$;
$\otimes $ represents the Kronecker product; 
$\textbf{B}(R)$ is a ball centered at the origin of radius $R\!\!\in\!\!\R_+$;
$\eta_i\!\!\in\!\!\R^n$ ({\it resp.,} $\xi_i\!\!\in\!\!\R^N$) is the $i_{th}$ element of the canonical Euclidean basis in $\R^n$ ({\it resp.,} $\R^N$);
let $x\!\!=\!\!(x_1,\dots,x_n)^\top\!\!\!\!\in\!\!\R^n$, $\|x\|$ is a norm in $\R^n$; $\|x\|_{P}\!\!=\!\!\sqrt{x^{\top}\!\!Px}$,  $P\!\!\in\!\! \R^{n\times n}$ satisfies $P\!\!\succ\!\!0$; 
$\|x\|_{\infty}\!\!=\!\!\max_{i=\overline{1,n}}|x_i|$;
the function $\alpha\!:\!\R_{\geq0}\!\!\to\!\!\R_{\geq0}$ is said to be of \textit{class $\mathscr{K}$} if  it is continuously strictly increasing with $\alpha(0)\!\!=\!\!0$; the function $\beta\!:\!\R_{\geq0}\!\!\times\!\!\R_{\geq0}\!\!\to\!\!\R_{\geq0}$ is said to be of \textit{class $\mathscr{KL}$} if for each fixed $t$ the function $s\!\!\mapsto\!\! \beta(s,t)$ is of class $\mathscr{K}$ and for each fixed $s$ the function $t\!\!\mapsto\!\! \beta(s,t)$ is continuously strictly decreasing; $L^\infty(\R,\R^n)$ is the space of Lebesgue measurable essentially bounded functions $q\!:\!\R\!\!\to\!\!\R^n$ with a norm 
$\|q\|_{L^\infty}\!\!:=\!\!\mathrm{ess} \sup_{t\in\R_+}\!\!\|q(t)\|_\infty\!\!<\!\!+\infty$. 
\section{Problem Formulation}\label{PF}

Consider a MAS composed of $N\!\!+\!\!1$ agents, characterized by the agent dynamics, the transmitted data, and the communication network, as Fig. \ref{fig:transmission}. 

\subsubsection{Agent dynamics}
the agent dynamics is as follows
\begin{equation}\label{eq:dynamic}
\begin{aligned}
 &\dot x_{i}(t)\!\!=\!\!Ax_i(t)\!\!+\!\!Bu_i(t),\;\; i\!\!=\!\!\overline{0,N}, \;\;t\!\!\in\!\!\R_+ \\
 &A\!\!=\!\!\left(\begin{smallmatrix}
      \textbf{0}_{n-1} & I_{n-1}  \\
       0  & \textbf{0}_{n-1}^\top
 \end{smallmatrix}\right), \quad B\!\!=\!\!\left(\begin{smallmatrix}
      \textbf{0}_{n-1}   \\
       1  
 \end{smallmatrix}\right),
\end{aligned}
\end{equation}
where $x_0\!\!\in\!\!\R^n$ is the state of the leader, with $u_0\!\!=\!\!0$; and $x_i\!\!\in\!\!\R^n$, $i\!\!=\!\!\overline{1,N}$ be the state of followers, with $u_i\!\!\in\!\!\R$, $i\!\!=\!\!\overline{1,N}$ be the control input to be designed.

\subsubsection{Transmitted data} the transmitted data of agent $i$ is $(x_i,v_i)\!\!\in\!\!\R^{n}\!\!\times\!\!\R^n$, where $v_i\!\!:=\!\!f_{v_i}(x_i,x_{j_1},\dots,x_{j_{n_i}},v_{j_1},\dots,v_{j_{n_i}})$, $f_{v_i}\!:\!\R^{n}\!\!\times\!\!\R^{n_in}\!\!\!\times\!\!\R^{n_in}\!\!\to\!\!\R^{n}$, $n_i$ is the number of agents which could convey information to agent $i$.

\begin{remark}
  If $v_i$, $\forall i\!\!=\!\!\overline{0,N}$ be unavailable, the setting corresponds to some classical configurations of MAS \cite{olfati2004consensus,zhu2023robust,9285180}; a quasi-decentralized configuration is achieved by letting $v_0\!\!=\!\!0$, $v_1(t)\!\!=\!\!\ldots\!\!=\!\!v_N(t)\!\!=\!\!u_g(t)$, $u_g\!:\!\R\!\!\to\!\!\R$ \cite{li2023generalized}; 
in addition, \cite{ning2020bipartite, hong2006tracking,ren2007multi,liu2011synchronization} report a distributed design with $v_{j_l}\!\!=\!\!u_{j_l}$, $l\!\!=\!\!\overline{1,{n_i}}$, $i\!\!=\!\!\overline{1,N}$. The extension of the transmitted data with the vector $v_i$ helps agents gather more information useful for control design.  
\end{remark}

\subsubsection{Communication network}
the communication network is described by a fixed and directed graph, and denoted by $\mathcal{G}\!\!=\!\!\{\mathcal{V},\mathcal{E},\mathcal{W}\}$. Here, $\mathcal{V}\!\!\!=\!\!\!\{\overline{0,N}\}$ represents the set of vertices, with vertex $0$ (the leader) serving as the root and vertices $\overline{1,N}$ (followers) serving as the terminals of at least one directed path originating from the root. Detailed definitions on $\mathcal{V}$, $\mathcal{E}$, and $\mathcal{W}$ are presented further on (Section \ref{GT}). Regarding this graph, we have $v_0\!\!=\!\!\textbf{0}_n$.

\begin{figure}[htbp]
    \centering
    \includegraphics[width=0.7\linewidth]{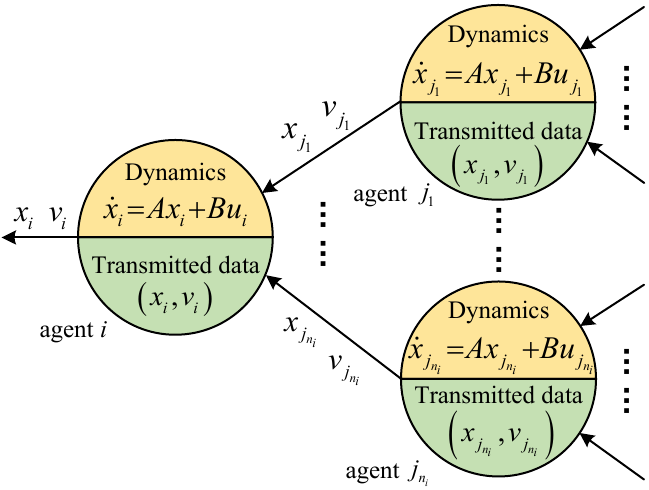}
    \caption{The configuration of the MAS: agent $i$, $i\!\!=\!\!\overline{1,N}$ receives the transmitted data $(x_{j_l},v_{j_l})\!\!\in\!\!\R^n\!\!\times\!\!\R^n$, $l\!\!=\!\!\overline{1,n_i}$, and send out $(x_i,v_i)\!\!\in\!\!\R^n\!\!\times\!\!\R^n$. }
    \label{fig:transmission}
\end{figure}

\subsubsection{Control problem}
regarding the MAS defined above, we introduce some new definitions based on the classical leader-following consensus introduced in \cite{ni2010leader}. 

\begin{Def}\label{def:lf_consensus}
    \textit{Let $e\!\!=\!\!(e_1^\top\!\!,\dots,e_N^\top\!)^\top$ with 
$e_i\!\!=\!\!x_i\!\!-\!\!x_0$  be the consensus error.
The MAS \eqref{eq:dynamic} with some  $u_i$, $i\!\!=\!\!\overline{1,N}$ achieves}
\\
    a) \textit{finite-time \cite{wang2010finite} (resp., asymptotic \cite{ni2010leader}) consensus  if $\lim_{t\!\to T}\|e(t)\|\!\!=\!\!0$, $\forall e(0)\!\!\in\!\!\R^{Nn}$, $T\!\!=\!\!T_c(e(0))$  (resp., $T\!\!=\!\!+\infty$);}
    \\
b) \textit{finite-time (resp., asymptotic) non-overshooting consensus if the finite-time (resp., asymptotic) consensus is achieved and there exists a non-empty set $\Omega\!\!\subset\!\!\Sigma\!\!:=\!\!\{e\!\!\in\!\!\R^{Nn}\!:\!(I_N\!\!\otimes\!\!\eta_1^\top)e\!\!\leq\!\!{0}\}$ being strictly positively invariant for the error system.
            }
\end{Def}
\begin{remark}
  The positive invariance of $\Omega$ guarantees the non-overshooting stabilization of $e$ if $e(0)\!\!\in\!\!\Omega$. In other words, if the follower’s first state component is initially behind the leader, this condition guarantees this behindness in the transient phase of reaching consensus.   
\end{remark}

The problem of finite-time/asymptotic non-overshooting consensus is inspired by the practical reason of safety-critical control systems \cite{ames2016control}, with its notion drawn from the theory of non-overshooting stabilizers \cite{krstic2006nonovershooting,polyakov2023finite}. To address this problem, this paper  extends the concept of non-overshooting (safety) control to MAS by a proper design of a transmitted vector $v_i\!\!=\!\!f_{v_i}(x_i,x_{j_1},\dots,x_{j_{n_i}},v_{j_1},\dots,v_{j_{n_i}})$, $i\!\!=\!\!\overline{1,N}$ and a distributed control protocol $u_i\!\!=\!\!f_{u_i}(v_i)$, $f_{u_i}\!:\!\R^{n}\!\!\to\!\!\R $,
$i\!\!=\!\!\overline{1,N}$ such that the closed-loop error equation 
\begin{equation}\label{eq:nonovershooterrordy}
    \dot e\!\!=\!\! f(e),\;\; f\!:\!\R^{Nn}\!\!\to\!\!\R^{Nn}
\end{equation}
is globally finite-time\footnote{The system $\dot{x}(t)\!\!=\!\!f(x(t))$, $t\!\!\in\!\!\R_+$, $x(0)\!\!=\!\!x_0\!\!\in\!\! \R^n,$ is globally finite-time stable \cite{orlov2004finite} if it is Lyapunov stable \cite{khalil2009lyapunov} and there exists a settling-time function $T(x_0)$, $T\!\!:\!\!\mathbb{R}^{n}\!\!\to\!\!\mathbb{R}_+$ such that $\|x(t)\|\!\!=\!\!0$,  $\forall t\!\!\geq\!\! T(x_0)$. } (asymptotically) stable, and properly
designed a non-empty set $\Omega\!\!\in\!\!\Sigma$  is strictly positively invariant for the error equation \eqref{eq:nonovershooterrordy}. 

\begin{remark}
    The consensus control protocol obtained in resolving the above-defined problem is valid for the MAS with nonlinear agent dynamics 
    \begin{equation*}
    \left\{\begin{array}{ll}
        \!\!\!\dot x_i\!\!=\!\!f_s(x_i,u_i),  \\
         \!\!\!y_i\!\!=\!\!f_o(x_i,u_i),
    \end{array}\right.f_s\!\!:\!\!\R^n\!\!\times\!\!\R^m\!\!\to\!\!\R^n, \;\;f_o\!\!:\!\!\R^n\!\!\times\!\!\R^m\!\!\to\!\!\R^o,\;\;i\!\!=\!\!\overline{0,N}
    \end{equation*}
    which is topologically equivalent \cite{isidori1985nonlinear} to the dynamics \eqref{eq:dynamic}. 
\end{remark}

\section{Preliminaries}\label{Pre}
\subsection{Elements of Graph Theory}\label{GT}
A fixed directed graph $\mathcal{G}$ is characterized by the   \textit{vertex set} $\mathcal{V}$, the \textit{edge set}
$\mathcal{E}$ and the \textit{weighted adjacency matrix}
$\mathcal{W}$. Specifically, the vertex set
$\mathcal{V}\!\!=\!\!\{\overline{1,N}\}$ is a collection of all nodes at the graph indexed by $i\!\!=\!\!\overline{1,N}$. The edge set
$\mathcal{E}\!\!=\!\!\{(i,j)|i,j\!\!\in\!\! \mathcal{V}\}$, $(i,j)\!\!\in\!\! \mathcal{E}$ if node $j$ could transfer its local information to node $i$; $n_i$ denotes the number of incoming edges of node $i$. The weighted adjacency matrix
$\mathcal{W}\!\!=\!\!\{w_{ij}\}\!\!\in\!\!\R^{N\!\times\!N}$, with $w_{ij}$, $i,j\!\!\in\!\!\mathcal{V}$, $w_{ij}\!\!=\!\!1$ if $(i,j)\!\!\in\!\! \mathcal{E}$, and $w_{ij}\!\!=\!\!0$ otherwise. 
The \textit{self-loop} is excluded in this paper, \textit{i.e.,} $w_{ii}\!\!=\!\!0$.
The \textit{Laplacian matrix} associated to the graph $\mathcal{G}$ is defined as $\mathcal{L}\!\!=\!\!\{l_{ij}\}\!\!\in\!\!\R^{N\!\times\!N}$, where $l_{ij}\!\!=\!\!-w_{ij}$ for $i\!\!\neq \!\!j$, and $l_{ij}\!\!=\!\!\sum^N_{k=1}\!\!w_{ik}$ for $i\!\!=\!\!j$.
The latter immediately implies that $\sum_{j=1}^N\!\!l_{ij}\!\!=\!\!0$, $\forall i\!\!=\!\!\overline{1,N}$.
A \textit{directed path} from vertex $j$ to vertex $i$ is formulated if 
 there is a sequence of vertices $\overline{i_1, i_s}$, with $i_1\!\!=\!\!j$, $i_s\!\!=\!\!i$ and $(i_{\iota+1},i_\iota)\!\!\in\!\!\mathcal{E}$, $\iota\!\!=\!\!\overline{1,s\!\!-\!\!1}$, $2\!\!\leq\!\!s\!\!\in\!\!\mathbb{N}_+$. 
 If there exists a vertex $i_r$ has at least one directed path from itself to any one of the rest nodes while $(i_r,j)\!\!\notin\!\!\mathcal{E}$, $\forall j\!\!\in\!\!\mathcal{V}$,  
we call $i_r$ the \textit{root}. 
The graph connected and acyclic is called the \textit{tree}.
The tree with a unique root and where every other vertex has exactly one incoming edge is called the \textit{directed spanning tree}. The directed graph with a unique root contains a directed spanning tree, and the associated Laplacian matrix has one simple zero eigenvalue while all others have positive real parts \cite{ren2005consensus}.
\begin{lem}\label{lem:unique}
    \textit{Let $\mathcal{G}_{lf}\!\!=\!\!\{\mathcal{V}_{lf},\mathcal{E}_{lf},\mathcal{W}_{lf}\}$ denote a directed graph with a unique root labeled by $0$, where $\mathcal{V}_{lf}\!\!=\!\!\{\overline{0,N}\}$, $\mathcal{W}_{lf}\!\!=\!\!\{w_{ij}\}\!\!\in\!\!\R^{(N\!+\!1)\!\times\!(N\!+\!1)}$, and
the associated Laplacian matrix be  $\mathcal{L}\!\!=\!\!\{l_{ij}\}\!\!\in\!\!\R^{(N\!+\!1)\!\times\!(N\!+\!1)}$. Then }
\\
 a)\textit{
there exists an invertible matrix $E\!\!\in\!\!\R^{(N+1)\!\times\!(N+1)}$ such that 
$$E^{-1}\!\!\mathcal{L}E
    \!\!=\!\!\left(\begin{smallmatrix}
    0 & \textbf{0}_N^\top\\
    \textbf{0}_N & \TL
    \end{smallmatrix}\right),$$ with $\TL\!\!\in\!\!\R^{N\!\times\!N}$ being invertible.}
\\ 
b)\textit{ the algebraic equation $\omega_0\!\!=\!\!\textbf{0}_m$,
 \begin{equation}\label{eq:lin_alge}
 	\omega_i\!\!=\!\!\tfrac{1}{\sum_{j=0}^N\!\!w_{ij}}\sum\nolimits_{j=0}^N\!\! w_{ij}\!\!\left(M(x_i\!\!-\!\!x_j)\!\!+\!\!\omega_j\right), \;M\!\!\in\!\! \R^{m\times n}, \omega_i\!\!\in\!\!\R^m,i\!\!=\!\!\overline{1,N}
 \end{equation}
has a unique solution, where $w_{ij}\!\!\in\!\!\R$ is the element of the matrix $\mathcal{W}_{lf}$, and $x_i$, $i\!\!=\!\!\overline{0,N}$ is the state vector of MAS \eqref{eq:dynamic}.  }
\end{lem}

\begin{proof}
Claim a) is obtained by referring the Lemma 3 of \cite{guan2012impulsive} with $l_{0j}\!\!=\!\!0$, $\forall j\!\!=\!\!\overline{0,N}$ since the root is labeled by $0$. 

For claim b), rewrite \eqref{eq:lin_alge}, with $\omega_{0}\!\!=\!\!\textbf{0}_m$, as 
 \begin{equation*}\begin{aligned}
  \sum\nolimits_{j=0}^Nw_{ij}\omega_i\!\!-\!\!\sum\nolimits_{j=1}^N w_{ij}\omega_j\!\!&=\!\!-\!\!\sum\nolimits_{j=0}^N w_{ij}M(x_j\!\!-\!\!x_i).
  \end{aligned}\end{equation*}
The latter is equivalent to $\TL_i \omega\!\!=\!\!(\TL_i\!\otimes\!M) e$, 
yielding a compact form $\TL \omega\!\!=\!\!(\TL\!\!\otimes\!\!M) e$, $\omega\!\!=\!\!(\omega_1^\top,\dots, \omega_N^\top)^\top$, $\TL_i$ is the $i_{th}$ line of $\TL$. Since $\TL$ is invertible the latter algebraic equation has a unique solution $\omega\!\!=\!\!(I_N\!\!\otimes\!\!M) e$. The proof is completed. 
\end{proof}

\subsection{Elements of Homogeneity Theory}
Homogeneity describes the \textit{symmetry} with respect to a transformation called \textit{dilation}. For example, let a function $h\!:\!\R^n\!\!\to\!\!\R$, whose homogeneity is represented as $\exists\mu\!\!\in\!\!\R$ such that $h(\exp(s)x)\!\!=\!\!\exp(\mu s)h(x)$, $\forall s\!\!\in\!\!\R$. In this example, the argument $x$ is uniformly scaled by $\exp(s)$, which is the simplest case of dilation. However, consider a non-uniform scaling on $x$, which can be denoted as $x\!\!\to\!\!\dn(s)x$, where $\dn(s)$ is an operator maps $\R^n$ to $\R^n$, $\forall s\!\!\in\!\!\R$. To be a dilation, $\dn(s)$ has to satisfy:
\\
$\bullet$ \textit{group property}: $\dn(0)\!\!=\!\!I_n$, and $\dn(a\!\!+\!\!b)\!\!=\!\!\dn(a)\dn(b)$, $a,b\!\!\in\!\!\R$;
\\
$\bullet$ \textit{limit property}: $\lim_{s\rightarrow \pm \infty}\|\mathbf{d}(s)x\|\!\!=\!\!\exp(\pm \infty)$, $x\!\!\in\!\!\R^n\backslash\{ \mathbf{0}\}$.\\
In many cases, we require the continuity of $\dn(s)$, which is determined by the continuity of function $s\!\!\mapsto\!\!\dn(s)x$, $\forall x\!\!\in\!\!\R^n$. We say $\dn(s)$ is a \textit{monotone dilation} if the latter function is strictly increasing \cite{polyakov2020generalized}. The \textit{linear dilation} is defined as
$\mathbf{d}(s)\!\!=\!\!\exp(sG_\dn)$, $s\!\!\in\!\! \mathbb{R}$, $G_\dn\!\!\in\!\!\R^{n\times n}$ is anti-Hurwitz and known as the \textit{generator} of the dilation. Particularly, if
$G_\dn\!\!=\!\!\diag\{r_i\}_{i=1}^n$, $r_i\!\!\in\!\!\R_+$, $i\!\!=\!\!\overline{1,n}$, the generated dilation is called the \textit{weighted dilation}. 
%
The following result supports us in determining a proper $G_\dn\!\!\in\!\!\R^{n\times n}$ for a continuous linear control system $\dot x\!\!=\!\!Ax\!\!+\!\!Bu$, $x\!\!\in\!\!\R^n$, $A\!\!\in\!\!\R^{n\times n}$, $B\!\!\in\!\!\R^{n\times m}$, $u\!\!\in\!\!\R^{m\times n}$.

\begin{lem}\cite{polyakov2020generalized}
\textit{Given $\mu\!\!\neq\!\! 0$, a nilpotent $A\!\!\in\!\!\R^{n\times n}$, and a matrix $B\!\!\in\!\!\R^{n\times m}$ such that $(A,B)$ controllable, then there always exists an anti--Hurwitz matrix $G_{\dn}$ satisfying:
\begin{equation}\label{eq:G_d}
	AG_{\dn}\!\!=\!\!(\mu I_n\!\!+\!\!G_{\dn})A, \quad G_{\dn}B\!\!=\!\!B.
\end{equation}
}
\end{lem}
If $(A,B)$ takes a canonical form as \eqref{eq:dynamic}
then the equation \eqref{eq:G_d} has an explicit solution $G_\dn\!\!=\!\!\diag\{1\!\!-\!\!\mu(n\!\!-\!\!k)\}_{k\!=\!1}^n$, $\mu\!\!<\!\!\tfrac{1}{n-1}$.
Based on the concept of dilation, the homogeneity is defined below.
\begin{Def}\cite{polyakov2020generalized}
	\textit{A vector field $f\!:\! \mathbb{R}^n \!\!\to\!\! \mathbb{R}^n$ (resp., function $h\!:\!\R^n\!\!\to\!\!\R$) is said to be $\dn$-homogeneous if there exists a $\mu\!\! \in\!\! \mathbb{R}$ such that
	\begin{equation*}\begin{aligned}
 \begin{array}{cc}
      f(\dn(s) x)\!\!=\!\!\exp(\mu s) \dn(s) f(x),\;\; \forall s\!\!\in\!\! \mathbb{R},\;\; \forall x \!\!\in\!\! \mathbb{R}^n  \\
      (resp.,\;h(\dn(s) x)\!\!=\!\!\exp(\mu s) h(x),\;\; \forall s\!\!\in\!\! \mathbb{R},\;\; \forall x \!\!\in\!\! \mathbb{R}^n) 
 \end{array}
\end{aligned}\end{equation*}
where $\dn$ is a dilation, $\mu$ is known as the homogeneity degree.}
\end{Def}

One main feature of homogeneous systems is the improved convergence rate characterized by its homogeneity degree $\mu$. 
\begin{lem}\label{lem1}
\textit{Let $f\!:\!\R^n\!\!\to\!\!\R^n$ be a continuous $\mathbf{d}$-homogeneous vector
field of degree $\mu\!\! \in\!\! \mathbb{R}$, and the origin of the system $\dot{x}\!\!=\!\!f(x)$ be globally asymptotically stable. Then the origin of the system $\dot{x}\!\!=\!\!f(x)$ is globally finite-time stable for $\mu\!\!\in\!\!\R_-$.}
\end{lem}

Another important feature of homogeneous systems is the robustness, in the sense of ISS\footnote{Input-to-State Stability. The system $\dot x(t)\!\!=\!\!f(x(t),q(t))$, $x\!\!\in\!\! \R^n$, $q\!\!\in\!\! \R^p$, $t\!\!\in\!\!\R_+$, $x(0)\!\!=\!\!x_0$ is ISS if there exist $\beta\!\!\in\!\! \mathscr{KL}$ and $\gamma\!\!\in\!\!\mathscr{K}$ such that
	\begin{equation*}
	    \|x(t)\|\!\!\leq\!\! \beta\left(\|x(0)\|,t\right)\!\!+\!\!\gamma\left(\|q\|_{L^{\infty}}\right),\;\;
	 \forall x(0)\!\!\in\!\! \R^n, \;\;\forall q\!\!\in\!\! L^{\infty}(\R,\R^p).
\end{equation*}}, with respect to disturbances.
\begin{lem}\cite{polyakov2020generalized}\label{lem:ISS}
\textit{Assume a continuous vector field $f_q\!:\!\mathbb{R}^n\!\!\times\!\!\R^p\!\!\to\!\! \mathbb{R}^{n+p}$ takes the form of
\begin{equation*}
  {f}_q(x,q)\!\!=\!\!\left(
                   \begin{smallmatrix}
                     f(x,q) \\
                     \mathbf{0}_p \\
                   \end{smallmatrix}
                 \right),\;\; x\!\!\in\!\! \mathbb{R}^{n}, \;\; q\!\!\in\!\! L^{\infty}\!(\R,\R^p),\;\;f\!:\!\R^n\!\!\times\!\!\R^p\!\!\to\!\!\R^n
\end{equation*}
and $\dn$-homogeneous of degree $\mu$ with respect to a dilation $\dn\!\!=\!\!\left(
                                             \begin{smallmatrix}
                                               \dn_x & \mathbf{0}_{n\!\times\! p} \\
                                               \mathbf{0}_{p\!\times\! n} & \dn_q \\
                                             \end{smallmatrix}
                                           \right)
$
in $\mathbb{R}^{n+p}$. If the system
$  \dot x\!\!=\!\!f(x,\mathbf{0})$ 
is asymptotically stable at the origin, then the system
$
  \dot x\!\!=\!\!f(x,q)
  $
is ISS.}
\end{lem}

Next, we define the \textit{canonical homogeneous norm}.
\begin{Def}\label{def:hnorm}
	\textit{The functional $\|\cdot\|_{\mathbf{d}}\!:\!\R^n\!\!\to\!\! \R_{\geq0}$ defined as $\|\mathbf{0}\|_{\dn}\!\!=\!\!0$ and
$\|x\|_\mathbf{d}\!\!=\!\!\exp(s_x)$, where $s_x\!\!\in\!\! \mathbb{R}\!:\!\|\mathbf{d}(-s_x)x\|\!\!=\!\!1$, $\mathbf{d}$ is a linear monotone dilation,
	is called the canonical homogeneous norm induced by a norm $\|\cdot\|$ in $\mathbb{R}^n$.}
\end{Def}

The canonical homogeneous norm is not a norm in the usual sense but indeed a norm in $\R^n_{\dn}$ being homeomorphic to $\R^n$ \cite{polyakov2020generalized}. If $\dn(s)\!\!=\!\!\exp(s)I_n$, then, $\|\cdot\|_{\dn}\!\!=\!\!\|\cdot\|$.
Moreover, 
the canonical homogeneous norm has the following properties:
\\
	$\bullet$ homogeneous of degree $1$, {\it i.e.,} $\|\mathbf{d}(s)x\|_\mathbf{d}\!\!=\!\!\exp(s)\|x\|_\mathbf{d}$,  $\forall s\!\!\in\!\! \mathbb{R}$;\\
	$\bullet$  $\|x\|\!\!=\!\!1 \!\!\Leftrightarrow\!\! \|x\|_{\dn}\!\!=\!\!1$;
 \\
	$\bullet$ $\|\cdot\|_{\dn}$ is a locally Lipschitz continuous function on $\R^n \backslash \{\mathbf{0}\}$;
 \\
	$\bullet$ if  $\|x\|\!\!=\!\!\|x\|_P$, where $P\!\!\in\!\!\R^{n\!\times\!n}$ satisfies
	$P\!\!\succ\!\! 0$, $PG_{\dn}\!\!+\!\!G_{\dn}^{\top}P\!\!\succ\!\!0$,
	then, the linear dilation $\dn(s)\!\!=\!\!\exp(sG_{\dn})$ is monotone and
	\begin{equation*}\label{eq:part-deriv}
		\tfrac{\partial \|x\|_{\dn}}{\partial x}\!\!=\!\!\tfrac{\|x\|_{\dn}x^{\top}\!\dn^{\top}\!(-\ln\!\|x\|_{\dn}) P\dn(-\ln \!\|x\|_{\dn})}{x^{\top}\!\dn^{\top}\!(-\ln \!\|x\|_{\dn})PG_{\dn}\dn(-\ln \!\|x\|_{\dn})x},\;\; x\!\!\in\!\!\R^n\backslash\{\mathbf{0}\}.
	\end{equation*}

\subsection{Distributed Homogeneous Consensus}\label{MR1}

In contrast to the quasi-decentralized homogeneous leader-following consensus control relying on a centralized design of $v_i$, $i\!\!=\!\!\overline{0,N}$ in \cite{li2023generalized}, we propose a distributed scheme.

\begin{thm}\label{thm:state_con}
\it{Let $\mu\!\!\in\!\![-1,\tfrac{1}{n-1})$. Let $(X,Y)\!\!\in\!\!\R^{n\!\times\!n}\!\!\times\!\!\R^{1\!\times\!n}$ be the solution of the following LMI:
\begin{equation}\label{eq:LMI_P_U}
 X\!\!\succ\!\!0,\quad G_\dn X\!\!+\!\!XG_\dn\!\!\succ\!\!0, \quad AX\!\!+\!\!X A^\top\!\!\!\!-\!\!BY\!\!-\!\!Y^\top\!\! B^\top\!\!\!\!\prec\!\!0,
\end{equation}
where $G_\dn\!\!=\!\!\diag\{1\!\!-\!\!\mu(n\!\!-\!\!k)\}_{k\!=\!1}^n$. Let $\dn$ be a dilation group generated by $G_\dn$. 
Let the transmitted vector 
\begin{equation}\label{eq:def_vi}\begin{aligned}
v_0\!\!=\!\!\textbf{0}_n,\quad v_i\!\!=\!\!\tfrac{1}{\sum_{j=0}^Nw_{ij}}\sum\nolimits_{j=0}^N w_{ij}\left(x_i\!\!-\!\!x_j\!\!+\!\!v_j\right), \;\; i\!\!=\!\!\overline{1,N}
 \end{aligned}\end{equation}
 and the distributed homogeneous consensus protocol
 \begin{equation}\label{eq:homo_control}
 	u_i\!\!=\!\!-\|v_i\|_{\dn}^{1+\mu}K\dn(-\ln\|v_i\|_{\dn})v_i,\;\; i\!\!=\!\!\overline{1,N}
 \end{equation}
with $K\!\!=\!\!YX^{-1}$, $\|\cdot\|_\dn$ is induced by the weighted Euclidean norm $\|\cdot\|_P$, $P\!\!=\!\!X^{-1}$,
such that 
\\
$\bullet$ equation \eqref{eq:def_vi}, \eqref{eq:homo_control} and \eqref{eq:nonovershooterrordy} have unique solution $v_i$, $u_i$ and $e_i$, $i\!\!=\!\!\overline{1,N}$, respectively;
\\
$\bullet$ closed-loop error equation \eqref{eq:nonovershooterrordy} is globally finite-time stable for $\mu\!\!\in\!\![-1,0)$.


}

\end{thm}

\begin{proof}
Recall Lemma \ref{lem:unique} with $\omega_i\!\!=\!\!v_i$, $ i\!\!=\!\!\overline{0,N}$, $M\!\!=\!\!I_n$, the equation \eqref{eq:def_vi} has a unique solution $v\!\!=\!\!e$, $v\!\!=\!\!(v_1^\top,\dots,v_N^\top)^\top$.
In this case, the closed-loop error equation \eqref{eq:nonovershooterrordy} 
yielding
\begin{equation}\label{eq:exp_error_dy}\begin{aligned}
\dot e\!\!=&f(e)\\
    \!\!=&\!\left(I_N\!\!\otimes\!\!A\!\!-\!\!\diag\{\|(\xi_i^\top\!\!\!\otimes\!\!I_n)e\|_{\dn}^{1+\mu}\!BK\dn(-\!\!\ln\!\|\!(\xi_i^\top\!\!\!\!\otimes\!\!I_n)e\|_{\dn})\}_{i=1}^N\right)e,
\end{aligned} \end{equation}
has a unique solution.
We notice equation \eqref{eq:exp_error_dy} is completely decoupled component-wise, \textit{i.e.,} the dynamics of $e_i$ can be written as $\dot e_i\!\!=\!\!f_{i}(e_i)\!\!=\!\!
  (A\!\!-\!\!\|e_i\|_{\dn}^{1+\mu}\!\!BK\dn(-\ln\|e_i\|_{\dn}))e_i$,  $f_{i}\!\!:\!\!\R^n\!\!\to\!\!\R^n$, $i\!\!=\!\!\overline{1,N}$.
In this case, the stability of $e_i\!\!\in\!\!\R^n$, $i\!\!=\!\!\overline{1,N}$ is exhibited by $e\!\!\in\!\!\R^{Nn}$. 

We first concentrate on the homogeneity.
Equation \eqref{eq:G_d} implies $(I_N\!\otimes\!A)\tilde\dn(s)\!\!=\!\!\exp(\mu s)\tilde\dn(s)(I_N\!\otimes\!A)$ and $\exp(s)(I_N\!\otimes\!B)\!\!=\!\!\tilde\dn( s)(I_N\!\otimes\!B)$ \cite{polyakov2020generalized}, where $\tdn\!\!=\!\!I_N\!\!\otimes\!\!\dn$.
The vector field $e\!\!\mapsto\!\!f(e)$ is $\tilde\dn$-homogeneous of degree $\mu\!\!\in\!\!\R$ with respect to $e\!\!\to\!\!\tilde\dn(s)e$ since
$f(\tilde\dn(s)e)\!\!=\!\!\exp(\mu s)\tilde\dn(s)f(e)$, $\forall s\!\!\in\!\!\R$.
Similarly, 
 $e_i\!\!\mapsto\!\!f_{i}(e_i)$ be $\dn$-homogeneous of degree $\mu\!\!\in\!\!\R$ with respect to $e_i\!\!\to\!\!\dn(s)e_i$ since $f_i(\dn(s)e)\!\!=\!\!\exp(\mu s)\dn(s)f_i(e)$, $\forall s\!\!\in\!\!\R$.


Then, we prove $\|e_i\|_\dn$ be the Lyapunov candidate for  $\dot e_i\!\!=\!\!f_{i}(e_i)$. Indeed,
\begin{equation*}\begin{aligned}
&\tfrac{d \|e_i\|_{\dn}}{d t}
\!\!=\!\!\tfrac{\|e_i\|_\dn^{1+\mu}e_i^{\top}\dn^{\top}\!\!\!(-\ln \|e_i\|_\dn) P(AX\!+\!XA^\top\!\!\!-\!BY\!-\!Y^\top\!\!\! B^\top)P\dn(-\ln \|e_i\|_\dn)e_i}
{e_i^{\top}\dn^{\top}\!\!\!(-\ln \|e_i\|_\dn)P(G_{\dn}X\!+\!XG_\dn )P\dn(-\ln \|e_i\|_\dn)e_i}.
	\end{aligned}\end{equation*}
We have $\tfrac{d \|e_i\|_{\dn}}{d t}\!\!<\!\!0$ since \eqref{eq:LMI_P_U} is satisfied. Since the latter analysis is valid for any $i\!\!=\!\!\overline{1,N}$, then we have the error equation $\dot e\!\!=\!\!f(e)$ being globally asymptotically stable. Finally,  the proof is completed by recalling Lemma \ref{lem1}. 
\end{proof}
For $\mu=-1$ the control \eqref{eq:homo_control} has a discontinuity at the point $v_i=\textbf{0}_{n}$. In this case, solutions of the closed-loop system are understood in the sense of  Filippov \cite{Filippov1988:Book}. 
Moreover, if some model uncertainty is introduced to the dynamics of agents \eqref{eq:dynamic} as below,
\begin{equation}\label{eq:dynamic_disturb}
\begin{aligned}
 \dot x_{i}(t)\!\!=\!\!Ax_i(t)\!\!+\!\!Bu_i(t)\!\!+\!\!q_{i}(t),\;\;  q_{i}\!\!\in\!\!L^\infty\!(\R,\R^{n}),\;\;i\!\!=\!\!\overline{0,N}.
\end{aligned}
\end{equation}
The error equation becomes
\begin{equation}\label{eq:exp_error_dy_disturb}\begin{aligned}
 	\dot e&\!\!=\!\!f(e,q)\\
  &\!\!=
    \!\!\!(I_N\!\!\otimes\!\!A)e
    \!\!-\!\!\diag\{\!\|\!(\xi_i^\top\!\!\!\!\otimes\!\!I_n)e\|_{\dn}^{1\!+\!\mu}\!BK\dn(\!-\!\!\ln\!\!\|\!(\xi_i^\top\!\!\!\!\otimes\!\!I_n)e\|_{\dn})\!\}_{i=1}^N\!e\!\!+\!\!q,
\end{aligned} \end{equation}
with $f\!:\!\R^{Nn}\!\!\times\!\!\R^{Nn}\!\!\to\!\!\R^{Nn}$, $q\!\!=\!\!((q_{1}\!\!-\!\!q_{0})^\top,\dots,(q_{N}\!\!-\!\!q_{0})^\top)^\top\!\!\!\!\!\in\!\!L^\infty\!(\R,\R^{Nn})$. 
\begin{Coro}\label{remark:ISS}
  {\it Let the conditions of Theorem \ref{thm:state_con} hold. Let $\mu\!\!\in\!\!(-1,0)$. The error equation \eqref{eq:exp_error_dy_disturb} is ISS with respect to $q\!\!\in\!\!L^\infty\!(\R,\R^{Nn})$.}
\end{Coro}
\begin{proof}
Recall Lemma \ref{lem:ISS}, let $f_q(e,q)\!\!=\!\!(f^\top\!\!(e,q),\textbf{0}^\top_{Nn})^\top$, $f_q\!\!:\!\!\R^{Nn}\!\!\times\!\!\R^{Nn}\!\!\to\!\!\R^{2Nn}$. It is validated that 
$f_q(\tdn(s)e,\exp(\mu s)\tdn(s)q)\!\!=\!\!\exp(\mu s)\diag\{\tdn(s),\exp(\mu s)\tdn(s)\}f_q(e,q)$, $\forall s\!\!\in\!\!\R$. Thus we have 
the vector field $f_q\!\!:\!\!\R^{Nn}\!\!\times\!\!\R^{Nn}\!\!\to\!\!\R^{2Nn}$ be  
homogeneous of degree $\mu$ with respect to dilation $(e,q)\mapsto (\tdn(s)e,\exp(\mu s)\tdn(s)q)$. The asymptotic stability of $\dot e\!\!=\!\!f(e,\textbf{0})$ is guaranteed by Theorem \ref{thm:state_con}.
Therefore, the claimed property is obtained.
\end{proof}

\section{Finite-time Non-overshooting Consensus Control Design}\label{MR2}


\subsection{Asymptotic Non-overshooting Consensus}
For a constant $\lambda\!\!\in\!\!\R_+$, let us introduce the row vectors
\begin{equation*}\begin{aligned}
    h_k\!\!=\!\!-\eta^\top_1\!\!(A\!\!+\!\!\lambda I_n)^{k-1},\;\; k\!\!=\!\!\overline{1,n}
\end{aligned}\end{equation*}
and consider a positive cone 
\begin{equation}\label{eq:lin_positive_cone}\begin{aligned}
    \Omega\!=\!\{e\!\in\!\R^{Nn}\!:\! (I_N\!\otimes\! H)e\!\geq\!{0}\}\!\subset\!\Sigma, \;\; H\!=\!(h^\top_1,\dots,h^\top_n)^\top\!, 
\end{aligned}\end{equation}
where $\Sigma\!\!=\!\!\{e\!\!\in\!\!\R^{Nn}\!:\!(I_N\!\!\otimes\!\!\eta_1^\top)e\!\!\leq\!\!{0}\}$.

\begin{thm}\label{thm:linear}
\textit{ Let $\lambda\!\!\in\!\!\R_+$. Let the transmitted vector $v_i$ be given by \eqref{eq:def_vi}, and the consensus control protocol be defined as
\begin{equation}\label{eq:lin_control}
    u_i\!\!=\!\!-K_{lin} v_i,\;\;
K_{lin}\!\!=\!\!\eta_1^\top\!\!(A\!\!+\!\!\lambda I_n)^n,\;\;i\!\!=\!\!\overline{1,N}.
\end{equation}
The control \eqref{eq:lin_control} asymptotically stabilizes the error equation \eqref{eq:nonovershooterrordy}, and renders $\Omega$ strictly positively invariant for the equation \eqref{eq:nonovershooterrordy}. Moreover, if
\begin{equation}\label{eq:lambda}\begin{aligned}
    \sum\nolimits_{\zeta=0}^{k-1}\!\!C_\zeta^{k-1}\!\!\lambda^\zeta\eta^\top_{k-\zeta}e_i(0)\!\!\leq\!\!{0},\;\; \forall k\!\!=\!\!\overline{2,n},\;\;\forall i\!\!=\!\!\overline{1,N},
\end{aligned}\end{equation}
with $C_\zeta^{k-1}\!\!=\!\!\tfrac{(k-1)!}{(k-1-\zeta)!\zeta!}$ then $e(0)\!\!\in\!\!\Omega$.} 
\end{thm}
\begin{proof}
Equation \eqref{eq:def_vi} has unique solution $v\!\!=\!\!e$ by recalling Theorem \ref{thm:state_con}, then \eqref{eq:lin_control} is compactly written as $u\!\!=\!\!-(I_N\!\otimes\!K_{lin}) e$, $u\!\!=\!\!(u_1,\dots,u_N)^\top\!\!\!\in\!\!\R^N$, and the error equation \eqref{eq:nonovershooterrordy} becomes
 \begin{equation}\label{eq:error_closed}\begin{aligned}
 	\dot e\!\!=\!\!f(e)
  \!\!=\!\!(I_N\!\!\otimes\!\!(A\!\!-\!\!BK_{lin}))e,\;\;e\!\!=\!\!(e_1^\top\!\!,\dots,e_N^\top\!)^\top
  \end{aligned}\end{equation}
  which is decoupled component-wise satisfying $\dot e_i\!\!=\!\!f_i(e_i)\!\!=\!\!(A\!\!-\!\!BK_{lin})e_i$, $f_i\!\!:\!\!\R^n\!\!\to\!\!\R^n$, $e_i\!\!=\!\!(e_{i,1},\dots,e_{i,n})^\top\!\!\!\in\!\!\R^n$.
Let the barrier function $\phi_{i}\!\!=\!\!(\phi_{i,1},\dots,\phi_{i,n})^\top$, $\phi_{i,k}\!:\!\R^n\!\!\to\!\!\R$ be
\begin{equation}\label{eq:varphidef}\begin{aligned}
    \phi_{i,1}\!\!=\!\!-e_{i,1},\quad \phi_{i,k}\!\!=\!\!\dot\phi_{i,k-1}\!\!+\!\!\lambda\phi_{i,k-1},\;\;k\!\!=\!\!\overline{2,n}\\
\end{aligned}\end{equation}
namely,
    $\phi_{i,k}\!\!=\!\!h_{k-1}\dot e_i\!\!+\!\!\lambda h_{k-1}e_i\!\!=\!\!h_{k-1}A e_i\!\!-\!\!h_{k-1} BK_{lin} e_i\!\!+\!\!\lambda h_{k-1}e_i$.
Since $h_{k-1}B\!\!=\!\!{0}$, then
$\phi_{i,k}\!\!=\!\!h_{k-1}(A\!\!+\!\!\lambda I_n)e_i
    \!\!=\!\!h_ke_i$.
Thus $\phi_i\!\!=\!\!He_i$, 
and can be compactly written as 
$\phi\!\!=\!\!(\phi_1^\top,\dots,\phi^\top_N)^\top\!\!\!\!=\!\!(I_N\!\!\otimes\!\!H)e$.
The positive orthant $\{\phi\!\!\in\!\!\R^{Nn}\!\!:\!\!\phi\!\!\geq\!\!{0}\}$  is identical with the positive cone $\Omega$. Furthermore, $I_N\!\otimes\!H$ is invertible if and only if $H$ is invertible.
Indeed, matrix $H$ is the observability matrix of the pair $(-\eta_1^\top,A\!\!+\!\!\lambda I)$, while yields zero-state observability {\it{i.e.,}} suppose a continuous linear system $\dot x\!\!=\!\!(A\!\!+\!\!\lambda I)x$, $x\!\!\in\!\!\R^n$, $y\!\!=\!\!-\eta_1^\top x\!\!\in\!\!\R$,  if  $y\!\!=\!\!0$, one has $x\!\!=\!\!\textbf{0}$. 
Thus $H$ and $I_N\!\!\otimes\!\!H$ are invertible, and 
$e\!\!=\!\!(I_N\!\!\otimes\!\!H^{-1})\phi$.
Besides, since $h_n B\!\!=\!\!-1$, then
\begin{equation*}\begin{aligned}
    \dot\phi_{i,n}\!\!=\!\!h_{n} \dot e_i
    \!\!=\!\!h_{n}A e_i\!\!-\!\!h_nBK_{lin} e_i
      \!\!=\!\!-\lambda h_n e_i
      \!\!=\!\!-\lambda\phi_{i,n}.
\end{aligned}\end{equation*}
Together \eqref{eq:varphidef} we have 
$    \dot\phi_i\!\!=\!\!(A\!\!-\!\!\lambda I_{n})\phi_i$,
and
\begin{equation}\label{eq:barr_dy}\begin{aligned}
    \dot\phi\!\!=\!\!(I_N\!\!\otimes\!\!(A\!\!-\!\!\lambda I_{n}) )\phi.
\end{aligned}\end{equation}
 Since $I_N\!\!\otimes\!\!(A\!\!-\!\!\lambda I_{n})$ is Metzler, then $\phi$-system \eqref{eq:barr_dy} is positive \cite{farina2011positive} and $\Omega\!\!=\!\!\{\phi\!\!\in\!\!\R^{Nn}\!:\!\phi\!\!\geq\!\!{0}\}$ be strictly positively invariant.

Next is to prove $\phi(e(0))\!\!=\!\!(\phi_1^\top(e_1(0)),\dots,\phi^\top_N(e_N(0)))^\top\!\!\!\geq\!\!{0}$. The latter inequality is guaranteed by $\phi_i(e_i(0))\!\!\geq\!\!{0}$, $\forall i\!\!=\!\!\overline{1,N}$, and $\phi_i(e_i(0))\!\!\geq\!\!{0}$ is guaranteed by $\phi_{i,k}(e_i(0))\!\!\geq\!\!{0}$, $\forall k\!\!=\!\!\overline{1,n}$.
Indeed, for any $k\!\!\geq\!\!2$, we have 
$    \phi_{i,k}(e_i(0))\!\!=\!\!h_ke_i(0)
    \!\!=\!\!\sum\nolimits_{\zeta=0}^{k-1}C_\zeta^{k-1}\lambda^\zeta h_1A^{k\!-\!1\!-\!\zeta}e_i(0)$.
Notice that for $\zeta\!\!=\!\!k\!\!-\!\!1$, we have $h_1A^0\!\!=\!\!-\eta_1^\top$; for $\zeta\!\!=\!\!k\!\!-\!\!2$, we have $h_1A\!\!=\!\!-\eta_2^\top$, \textit{etc.} Then,  
    $\phi_{i,k}(e_i(0))
    \!\!=\!\!-\!\!\sum\nolimits_{\zeta=0}^{k-1}C_\zeta^{k-1}\lambda^\zeta\eta^\top_{k-\zeta}e_i(0)$.
Therefore, $\phi_{i,k}(e_i(0))\!\!\geq\!\!0$, $\forall k\!\!=\!\!\overline{2,n}$, $\forall i\!\!=\!\!\overline{1,N}$  provided \eqref{eq:lambda} holds. The latter implies $\phi(e(0))\!\!\!\geq\!\!{0}$, {\it i.e.,} $e(0)\!\!\in\!\!\Omega$.

Finally, notice that equation \eqref{eq:barr_dy} is asymptotically stabilized to the origin. Recall $e\!\!=\!\!(I_N\!\!\otimes\!\!H^{-1})\phi$, equation \eqref{eq:error_closed} has solution $e(t)\!\!=\!\!(I_N\!\!\otimes\!\!H^{-1})\exp( I_{N}\!\!\otimes\!\!(A\!\!-\!\!\lambda I)t)(I_N\!\!\otimes\!\!H)e(0)$, which means the consensus error $e\!\!\in\!\!\R^{Nn}$ is asymptotically stable. 
\end{proof}

\begin{Coro}
{\it Given $e(0)\!\!\in\!\!\mathrm{int}\Sigma$, inequality \eqref{eq:lambda} always holds if $\lambda\!\!\in\!\!\R_+$ is large enough.} 
\end{Coro}
\begin{proof}
   $e(0)\!\!\in\!\!\mathrm{int}\Sigma$ implies $\phi_{i,1}(0)\!\!\geq\!\!0$, $\forall i\!\!=\!\!\overline{1,N}$. Besides, inequality \eqref{eq:lambda} can be written as 
\begin{equation*}\begin{aligned}
    -\lambda^{k-1}\!\!\phi_{i,1}(0)\!\!+\!\!\sum\nolimits_{\zeta=0}^{k-2}C_\zeta^{k-1}\lambda^\zeta\eta^\top_{k-\zeta}e_i(0)\!\!\leq\!\!{0},\;\; \forall k\!\!=\!\!\overline{2,n},\;\forall i\!\!=\!\!\overline{1,N},
\end{aligned}\end{equation*}
which always holds for a sufficient large $\lambda\!\!\in\!\!\R_+$.
\end{proof}


\subsection{Finite-time Non-overshooting Consensus }
 By a proper scaling of the control gain $K_{lin}$, the linear control given by \eqref{eq:lin_control} can be upgraded into a homogeneous one \cite{wang2021generalized,li2023generalized,polyakov2020generalized}. 
In this paper, a similar scheme is developed by using the special structure of the transmitted vector $v_i$, $i\!\!=\!\!\overline{1,N}$.

\begin{thm}\label{thm:finite_nonovershoot}
\it{Let $\mu\!\!\in\!\![-1,0)$. Let $P\!\!\in\!\!\R^{n\!\times\!n}$ be the solution of the following LMI:
\begin{equation}\label{eq:LMI_P}
\begin{aligned}
   P\!\!\succ\!\!0,\quad PG_\dn\!\!+\!\!G_\dn P\!\!\succ\!\!0, \quad P(A\!\!-\!\!BK_{lin})\!\!+\!\!(A\!\!-\!\!BK_{lin})^\top\!\!\! P\!\!\prec\!\!0,
\end{aligned}
\end{equation}
where $G_\dn\!\!=\!\!\diag\{1\!\!-\!\!\mu(n\!\!-\!\!k)\}_{k\!=\!1}^n$,  $K_{lin}$ is defined in \eqref{eq:lin_control}. Let $\dn$ be a dilation group generated by $G_\dn$. Let the transmitted vector $v_i\!\!\in\!\!\R^n$ be defined in \eqref{eq:def_vi},
and the homogeneous consensus control protocol
\begin{equation}\label{eq:homo_control_nonover}
 	u_i\!\!=\!\!-\|v_i\|_{\dn}^{1+\mu}K_{lin}\dn(-\ln\|v_i\|_{\dn})v_i,\;\; i\!\!=\!\!\overline{1,N}
 \end{equation}
with $\|\cdot\|_\dn$ be induced by the weighted Euclidean norm $\|\cdot\|_P$,
such that the error equation \eqref{eq:nonovershooterrordy} be globally finite-time stable and 
the homogeneous cone\footnote{A non-empty set $\Omega_\dn\!\!\in\!\!\R^n$ is said to be a homogeneous cone if $x\!\!\in\!\!\Omega_\dn\!\!\Rightarrow\!\!\dn(s)x\!\!\in\!\!\Omega_\dn$, $\forall s\!\!\in\!\!\R$.}
$$\Omega_{\dn}\!\!=\!\!\{e\!\!\in\!\!\R^{Nn}:\diag\{H\dn(-\ln \|e_i\|_{\dn})\}_{i=1}^Ne\!\!\geq\!\!{0}\}$$ be a strictly positively invariant set for the equation \eqref{eq:nonovershooterrordy}, and $e(0)\!\!\in\!\!\Omega_{\dn}\!\!\subset \!\!\Sigma$
if $e(0)\!\!\in\!\!\Omega$ and $\|e_i(0)\|_{ P}\!\!\leq\!\!1$, $\forall i\!\!=\!\!\overline{1,N}$, with $\Omega\!\!\subset\!\! \Sigma$  and $H$ are defined in \eqref{eq:lin_positive_cone}.
}

\end{thm}

\begin{proof}
Recall Theorem \ref{thm:state_con}, equation \eqref{eq:def_vi} has a unique solution $v\!\!=\!\!e$, then the equation \eqref{eq:homo_control_nonover} becomes
\begin{equation*}\begin{aligned}
u\!\!&=\!\!-\diag\{\|e_i\|_{\dn}^{1+\mu}\!\!K_{lin}\dn(\!\!-\!\!\ln\!\!\|e_i\|_{\dn})\}_{i=1}^Ne,\;\;u\!\!=\!\!(u_1,\dots,u_N)^\top\!\!\!\in\!\!\R^N.
\end{aligned} \end{equation*} 
Thus the closed-loop error equation \eqref{eq:nonovershooterrordy} becomes
\begin{equation}\label{eq:exp_error_nonover}\begin{aligned}
 	\dot e&\!\!=\!\!f(e)\\
    &\!\!=\!\!\!\left(\!\!I_N\!\!\otimes\!\!A\!\!-\!\!\diag\{\|(\xi_i^\top\!\!\!\!\otimes\!\!I_n)e\|_{\dn}^{1+\mu}\!\!BK_{lin}\dn(-\!\!\ln\!\!\|(\xi_i^\top\!\!\!\!\otimes\!\!I_n)e\|_{\dn})\}_{i=1}^N\!\!\right)\!\!e.
\end{aligned} \end{equation}
Similar with \eqref{eq:exp_error_dy}, the error equation \eqref{eq:exp_error_nonover} is component-wise decoupled and the stability of $e\!\!\in\!\!\R^{Nn}$ is equivalent to $e_i\!\!\in\!\!\R^n$, $i\!\!=\!\!\overline{1,N}$. Select $\|e_i\|_\dn$ as the Lyapunov candidate, and repeat the analysis in Theorem \ref{thm:state_con}, we can obtain the homogeneity of $e\!\!\mapsto\!\!f(e)$ and $e_i\!\!\mapsto\!\!f_i(e_i)$ as well as $\tfrac{d\|e_i\|_\dn}{dt}\!\!<\!\!0$, $i\!\!=\!\!\overline{1,N}$ since \eqref{eq:LMI_P} holds. Thus the homogeneous error equation 
\eqref{eq:exp_error_nonover} is globally finite-time stable for $\mu\!\!\in\!\![-1,0)$.

Consider the homogeneous barrier function $\phi_i\!\!:\!\!\R^{n}\!\!\to\!\!\R^{n}$ as 
$\phi_i(e_i)\!\!=\!\!H\dn(-\ln \|e_i\|_{\dn})e_i$. Then, 
\begin{equation*}\begin{aligned}
\tfrac{d\phi_i}{dt}\!\!=\!\!-\!\tfrac{\frac{d \|e_i\|_{\dn}}{dt}}{\|e_i\|_{\dn}}&HG_\dn\dn(-\ln \|e_i\|_{\dn})e_i\\
&\!\!+\!H\dn(-\ln \|e_i\|_{\dn})\!\!\!\left(
\!\!A\!\!-\!\!\|e_i\|_{\dn}^{1+\mu}\!BK_{lin}\dn(-\ln\|e_i\|_{\dn})\!\!\right)\!\!e_i.
\end{aligned}\end{equation*}
Reusing $\dn(s)A\!\!=\!\!\exp(-\mu s)A\dn(s)$, $\dn( s)B\!\!=\!\!\exp(s)B$, $\forall s\!\!\in\!\!\R$ has
\begin{equation*}\begin{aligned}
&\tfrac{d\phi_i}{dt}
\!\!=\!\!\\&-\!\tfrac{\frac{d \|e_i\|_{\dn}}{dt}}{\|e_i\|_{\dn}}HG_\dn\dn(-\ln \|e_i\|_{\dn})e_i\!\!
+\!\!\|e_i\|_{\dn}^\mu H(A\!\!-\!\!BK_{lin})\dn(-\ln\|e_i\|_{\dn})e_i,
\end{aligned}\end{equation*}
and straightforwardly
\begin{equation}\label{eq:add}\begin{aligned}
\tfrac{d\phi_i}{dt}
\!\!=\!\!\|e_i\|_{\dn}^\mu H(A\!\!-\!\!BK_{lin}\!\!+\!\gamma_i G_\dn)\dn(-\ln \|e_i\|_{\dn})e_i,
\end{aligned}\end{equation}
where $\gamma_i\!\!=\!\!-\|e_i\|_{\dn}^{-\mu-1}\!\frac{d \|e_i\|_{\dn}}{dt}\!\!>\!\!0$,  $ e_i\!\!\in\!\!\R^{n}\backslash\{\textbf{0}\}$, $\forall i\!\!=\!\!\overline{1,N}$.
According to Theorem \ref{thm:linear}, 
$H(A\!\!-\!\!BK_{lin})H^{-1}\!\!=\!\!A\!\!-\!\!\lambda I_{n}$, which implies $H(A\!\!-\!\!BK_{lin})\!\!=\!\!(A\!\!-\!\!\lambda I_{n})H$. Besides, 
$HG_\dn\!\!=\!\!H(I_n\!\!+\!\!\Pi)\!\!=\!\!H\!\!+\!\!(
   \Pi^\top   h_1^\top,
     \ldots,
    \Pi^\top h_n^\top  
)^\top$, $\Pi\!\!=\!\!\diag\{\!-\!\mu(n\!\!-\!\!k)\}_{k\!=\!1}^n$.
We notice 
$h_k  \Pi \!\!=\!\!-\mu(n\!\!-\!\!k)h_k\!\!-\!\!\lambda\mu(k\!\!-\!\!1) h_{k-1}$,
which indicates 
$
HG_\dn
\!\!=\!\! \Gamma H$
with $\Gamma\!\!=\!\!G_\dn \!\!+\!\lambda\diag\{-\!\mu(k\!\!-\!\!1)\}_{k=1}^n A^\top$.
Then \eqref{eq:add} becomes
\begin{equation}\label{eq:homo_phi_i}\begin{aligned}
\tfrac{d\phi_i}{dt}
\!\!=\!\!\|e_i\|_{\dn}^\mu(A\!\!-\!\!\lambda I_{n}\!\!+\!\!\gamma_i\Gamma)\phi_i.
\end{aligned}\end{equation}
Notice that $\Gamma\!\!\geq\!\!0$ for $\mu\!\!\in\!\![-1,0)$. This directly implies $A\!\!-\!\!\lambda I_{n}\!\!+\!\!\gamma_i\Gamma$ is Metzler.
Let $\phi_i$ be compactly presented as $\phi\!\!=\!\!(\phi_1^\top,\dots,\phi_N^\top)^\top\!\!=\!\!\diag\{H\dn(-\ln \|e_i\|_{\dn})\}_{i=1}^Ne$, with the dynamics
\begin{equation}\label{eq:homo_phi}\begin{aligned}
\tfrac{d\phi}{dt}
\!\!=\!\!(
     \tfrac{d^\top\!\!\!\phi_1}{dt},
     \ldots,
     \tfrac{d^\top\!\!\!\phi_N}{dt} 
)^\top\!\!\!\!=\!\!\diag\{\|e_i\|_{\dn}^\mu(A\!\!-\!\!\lambda I_{n}\!\!+\!\!\gamma_i\Gamma)\}_{i=1}^N\phi\!\!:=\!\!\Lambda\phi.
\end{aligned}\end{equation}
Thus $\phi$-system \eqref{eq:homo_phi} is positive since $\Lambda\!\!\in\!\!\R^{Nn\!\times\!Nn}$ being Metzler \cite{farina2011positive}, and the homogeneous cone $\Omega_{\dn}\!\!=\!\!\{\phi\!\!\in\!\!\R^{Nn}:\phi\!\!\geq\!\!{0}\}$ is a strictly positively invariant set, which means for $e(0)\!\!\in\!\!\Omega_{\dn}\!\!\Rightarrow\!\!e(t)\!\!\in\!\!\Omega_{\dn}$, $\forall t\!\!\in\!\!\R_+$. Moreover, $\Omega_{\dn}\!\!\subset\!\!\Sigma$ by construction.

Next, we are supposed to prove $e(0)\!\!\in\!\!\Omega_{\dn}$, {\it{i.e.,}} $\phi(e(0))\!\!\geq\!\!{0}$, with
$\phi(e(0))\!\!=\!\!\diag\{H\dn(-\ln\|e_i(0)\|_\dn)\}_{i=1}^Ne(0)$.
On the one hand, for all 
$i\!\!=\!\!\overline{1,N}$, $\|e_i(0)\|_P\!\!\leq\!\!1\!\!\Leftrightarrow\!\!\|e_i(0)\|_{\dn}\!\!\leq\!\!1\!\!\Leftrightarrow\!\!-\ln \|e_i(0)\|_{\dn}\!\!\geq\!\!0$ and $\|e_i(0)\|_P\!\!=\!\!1\!\!\Leftrightarrow\!\!\|e_i(0)\|_{\dn}\!\!=\!\!1\!\!\Leftrightarrow\!\!-\ln \|e_i(0)\|_{\dn}\!\!=\!\!0\!\!\Leftrightarrow\!\!\phi(e(0))\!\!=\!\!(I_N\!\!\otimes\!\!H)e(0)$. 
On the other hand, 
let $\phi_s:\R\!\!\to\!\!\R^{n}$ with an auxiliary function be
$\phi_s(s)\!\!=\!\!H\dn(s)x$, $s\!\!\in\!\!\R_{\geq0}$, $x\!\!\in\!\!\R^{n}$,
yields $\phi_s(0)\!\!=\!\!Hx\!\!\geq\!\!{0}$. We prove $\phi_s(s)$ is non-decreasing on $s\!\!\in\!\!\R_{\geq0}$. Indeed, 
$\tfrac{d\phi_s}{ds}
\!\!=\!\!H G_\dn\dn(s)x
\!\!=\!\!\Gamma \phi_s$,
with
$\Gamma\!\!=\!\!G_\dn \!\!+\!\lambda\diag\{-\!\mu(k\!\!-\!\!1)\}_{k=1}^n A^\top$.
Notice that for $\mu\!\!\in\!\![-1,0)$, $\Gamma$ is anti-Hurwitz since it is a lower-triangular matrix with positive diagonal entries.
Therefore,  
$\phi_s(s)\!\!\geq\!\!{0}$, $\forall s\!\!\in\!\!\R_{\geq0}$. 
Based on the above, we have $\phi_i(e_i(0))\!\!\geq\!\!{0}$ if 
$\|e_i(0)\|_P\!\!\leq\!\!1$ and $He_i(0)\!\!\geq\!\!{0}$, $\forall i\!\!=\!\!\overline{1,N}$, and it is concluded as $\phi(e(0))\!\!\geq\!\!{0}$ provided 
$\|e_i(0)\|_P\!\!\leq\!\!1$, $\forall i\!\!=\!\!\overline{1,N}$ and $e(0)\!\!\in\!\!\Omega$.
\end{proof}

\subsection{Robustness Analysis}\label{RA}
Now we move to the disturbed case by considering the model uncertainty defined in \eqref{eq:dynamic_disturb}. The error equation is \begin{equation}\label{eq:exp_error_dy_disturb_non}\begin{aligned}
 	&\dot e\!\!=\!\!f(e,q)\\
  &\!\!=
    \!\!\!\left(\!\!I_N\!\!\otimes\!\!A
    \!\!-\!\!\diag\{\|\!(\xi_i^\top\!\!\!\!\otimes\!\!I_n)e\|_{\dn}^{1\!+\!\mu}\!BK_{lin}\dn(\!-\!\!\ln\!\!\|\!(\xi_i^\top\!\!\!\!\otimes\!\!I_n)e\|_{\dn})\}_{i=1}^N\!\!\right)\!\!e\!\!+\!\!q.
\end{aligned} \end{equation}
 with $q\!\!=\!\!((q_{1}\!\!-\!\!q_0)^\top,\dots,(q_{N}\!\!-\!\!q_0)^\top)^\top\!\!\!\in\!\!L^\infty\!(\R,\R^{Nn})$.

 Before analysis, according to \cite{boyd1994linear}, we claim $P(A\!\!-\!\!BK_{lin})\!\!+\!\!(A\!\!-\!\!BK_{lin})^\top\!\!\! P\!\!\prec\!\!0$ defined in \eqref{eq:LMI_P} of Theorem \ref{thm:finite_nonovershoot} implies that there exists $\rho\!\!\in\!\!\R_+$ such that
 \begin{equation}\label{eq:LMI_new}
    P(A\!\!-\!\!BK_{lin})\!\!+\!\!(A\!\!-\!\!BK_{lin})^\top\!\!\! P\!\!+\!\!\rho P\!\!\prec\!\!0.
 \end{equation}

\begin{Coro}\label{prop:1}
    \textit{Let the conditions of Theorem \ref{thm:finite_nonovershoot} hold. Then,
    \\
    $a)$ for $\mu\!\!\in\!\!(-1,0)$, the closed-loop error equation \eqref{eq:exp_error_dy_disturb_non} is ISS with respect to $q\!\!\in\!\!L^\infty\!(\R,\R^{Nn})$;
    \\
    $b)$ for $\mu\!\!\in\!\!(-1,0)$, and $q\!\!\in\!\!L^\infty\!(\R,\R_{-}^{Nn})$, the homogeneous cone
     $\Omega_{\dn}\!\!=\!\!\{e\!\!\in\!\!\R^{Nn}\!:\!\diag\{H\dn(-\ln \|e_i\|_{\dn})\}_{i=1}^Ne\!\!\geq\!\!{0}\}$ is  strictly positively invariant for equation \eqref{eq:exp_error_dy_disturb_non};
    \\
    $c)$ for $n\!\!=\!\!2$, $\mu\!\!=\!\!-1$, and $q_i\!\!=\!\!B\hat q_i$, $\hat q_i\!\!\in\!\!L^\infty(\R,\R)$, $i\!\!=\!\!\overline{0,N}$
    the error equation \eqref{eq:exp_error_dy_disturb_non} is globally finite-time stable and the homogeneous cone $\Omega_{\dn}\!\!=\!\!\{e\!\!\in\!\!\R^{2N}\!:\!\diag\{H\dn(-\ln \|e_i\|_{\dn})\}_{i=1}^Ne\!\!\geq\!\!{0}\}$ is  strictly positively invariant  for the error equation \eqref{eq:exp_error_dy_disturb_non} if 
    \begin{equation}\label{eq:upperq}
     \|\hat q_i\!\!-\!\!\hat q_0\|_{L^\infty}\!\!\leq\!\!\min\left\{\tfrac{\rho}{2|P^{1/2}|},\tfrac{\lambda\vartheta\;\lambda_{\min}(P^{-1/2}H^\top HP^{-1/2})}{\lambda^{1/2}_{\max}(P^{-1/2}\eta_1\eta_1^\top P^{-1/2})}\right\},\;\;\forall i\!\!=\!\!\overline{1,N}
    \end{equation}
    with $\rho\!\!\in\!\!\R_+$ defined in \eqref{eq:LMI_new}, 
and $\vartheta\!\!=\!\!\tfrac{\rho}{2\lambda_{\max}(P^{-1/2}(PG_\dn \!+\!G_\dn P)P^{-1/2})}$.
}
    
\end{Coro}

\begin{proof}
Claim a) is obtained by recalling Corollary \ref{remark:ISS}.

In addition, since $q\!\!\in\!\!L^\infty\!(\R,\R_{-}^{Nn})$, claim a) implies $\tfrac{d\|e_i\|_\dn}{dt}\!\!<\!\!0$, $\forall i\!\!=\!\!\overline{1,N}$ on $e\!\!\in\!\!\R^{Nn}$. Recall Theorem \ref{thm:finite_nonovershoot}, $\Lambda\!\!\in\!\!\R^{Nn\!\times\!Nn}$ defined in \eqref{eq:homo_phi} is Metzler on $e\!\!\in\!\!\R^{Nn}$. 
Consider the homogeneous barrier function $\phi_i\!\!:\!\!\R^{n}\!\!\to\!\!\R^{n}$ as 
$\phi_i(e_i)\!\!=\!\!H\dn(-\ln \|e_i\|_{\dn})e_i$.
Then, 
\begin{equation}\label{eq:homo_phi_iss}\begin{aligned}
\tfrac{d\phi}{dt}&\!\!=\!\!\Lambda\phi
\!\!-\!\!(I_N\!\!\otimes\!\!H)\diag\{\dn(-\!\!\ln \!\!\|e_i\|_{\dn})\}_{i=1}^N \tilde{q},\;\;\phi\!\!=\!\!(\phi_1^\top,\dots,\phi_N^\top)^\top
\end{aligned}\end{equation}
where $\tilde q\!\!=\!\!-q\!\!\in\!\!L^\infty\!(\R,\R^{Nn}_{+})$. 
Recall the definition of $H$ in  \eqref{eq:lin_positive_cone}, the matrix $\!\!-\!(I_N\!\!\otimes\!\!H)\diag\{\dn(-\ln \|e_i\|_{\dn})\}_{i=1}^N\!\!\geq\!\!0$. Therefore, the $\phi$-system \eqref{eq:homo_phi_iss} is positive \cite{farina2011positive}, and claim b) is obtained.

For claim c), we have $q\!\!=\!\!(I_N\!\!\otimes\!\!B)(\hat q_1\!\!-\!\!\hat q_0,\ldots,\hat q_N\!\!-\!\!\hat q_0)^\top$ in \eqref{eq:exp_error_dy_disturb_non}. 
It is clear the stability of error equation \eqref{eq:exp_error_dy_disturb_non} is equivalent to the stability of $\dot e_i\!\!=\!\!f_i(e_i,B(\hat q_i\!\!-\!\!\hat q_0))$, $i\!\!=\!\!\overline{1,N}$, $f_i\!\!:\!\!\R^2\!\!\times\!\!\R^2\!\!\to\!\!\R^2$. Let $\|e_i\|_\dn$ be the Lyapunov candidate, whose derivative along $\dot e_i\!\!=\!\!f_i(e_i,B(\hat q_i\!\!-\!\!\hat q_0))$ be
\begin{equation*}\begin{aligned}
&\tfrac{d \|e_i\|_{\dn}}{d t}\\
&\!\!=\!\!\!\tfrac{e_i^{\top}\!\!\!\dn^{\top}\!\!\!(\!\!-\!\!\ln \!\|e_i\|_\dn\!) (P(A\!-\!BK_{lin})\!+\!(A\!-\!BK_{lin} )^\top\!\!\!P)\dn(\!\!-\!\!\ln \!\|e_i\|_\dn\!)e_i\!+\!e_i^{\top}\!\!\dn^{\top}\!\!\!(\!\!-\!\!\ln \!\|e_i\|_\dn\!) PB(\hat q_i\!-\!\hat q_0) }
{e_i^{\top}\dn^{\top}\!\!\!(-\ln \|e_i\|_\dn)(PG_{\dn}+G_\dn P)\dn(-\ln \|e_i\|_\dn)e_i}.
	\end{aligned}\end{equation*}
It is noticeable that, in the latter equation, $e_i^{\top}\!\!\!\dn^{\top}\!\!\!(\!\!-\!\!\ln \!\|e_i\|_\dn\!) (P(A\!\!-\!BK_{lin})\!+\!(A\!-\!BK_{lin} )^\top\!\!\!P)\dn(\!\!-\!\!\ln \!\|e_i\|_\dn\!)e_i\!\!<\!\!-\rho$ since \eqref{eq:LMI_new} holds and $\|\dn(\!\!-\!\!\ln \!\|e_i\|_\dn\!)e_i\|_P\!\!=\!\!1$. In addition, $e_i^{\top}\!\!\dn^{\top}\!\!\!(\!\!-\!\!\ln \!\|e_i\|_\dn\!) PB(\hat q_i\!-\!\hat q_0)\!\!\leq\!\!\tfrac{\rho}{2}$ since \eqref{eq:upperq} holds.  Then $\tfrac{d \|e_i\|_{\dn}}{d t}\!\!\leq\!\!-\vartheta$, $\forall i\!\!=\!\!\overline{1,N}$, and the error equation \eqref{eq:exp_error_dy_disturb_non} is globally finite-time stable \cite{polyakov2020generalized}.
Moreover, since $\dn( s)B\!\!=\!\!\exp(s)B$, $\forall s\!\!\in\!\!\R$, $HB\!\!=\!\!-B$, the homogeneous barrier function  
$\phi_i(e_i)\!\!=\!\!H\dn(-\ln \|e_i\|_{\dn})e_i$ yields
\begin{equation}\label{eq:homo_phi_iss_mu-1}\begin{aligned}
\tfrac{d\phi}{dt}
\!\!=\!\!\diag\left\{\|e_i\|_{\dn}^{-1}\left(\begin{smallmatrix}
    2\gamma_i\!-\!\lambda &  1\\
   \lambda\gamma_i\!-\!\tfrac{(\hat q_i\!-\!\hat q_0)\phi_{i,1}}{|\phi_i|^2}  & \gamma_i\!-\!\lambda\!-\!\tfrac{(\hat q_i\!-\!\hat q_0)\phi_{i,2}}{|\phi_i|^2}
\end{smallmatrix}\right)\right\}_{i=1}^N\!\!\phi
\end{aligned}\end{equation}
with $\phi\!\!=\!\!(\phi_1^\top,\dots,\phi_N^\top)^\top$, 
$\gamma_i\!\!=\!\!-\!\tfrac{d \|e_i\|_{\dn}}{dt}\!\!>\!\!0$,
$\forall i\!\!=\!\!\overline{1,N}$. Moreover, the system matrix of the $\phi$-system \eqref{eq:homo_phi_iss_mu-1} is Metzler if
\begin{equation}\label{eq:Metz}
    \lambda\gamma_i\!\!-\!\!\tfrac{(\hat q_i-\hat q_0)\phi_{i,1}}{|\phi_i|^2}\!\!\geq\!\!0, \;\;\forall i\!\!=\!\!\overline{1,N}
\end{equation}
which holds if $\|\hat q_i\!\!-\!\!\hat q_0\|_{L^\infty}\!\!\leq\!\!\tfrac{\lambda\gamma_i|\phi_i|^2}{\phi_{i,1}}$, $\forall i\!\!=\!\!\overline{1,N}$. We notice $\gamma_i\!\!\geq\!\!\vartheta$, $\forall i\!\!=\!\!\overline{1,N}$, 
 $\lambda^{1/2}_{\min}(P^{-1/2}\eta_1\eta_1^\top P^{-1/2})\!\!\leq\!\!\phi_{i,1}\!\!\leq\!\!\lambda^{1/2}_{\max}(P^{-1/2}\eta_1\eta_1^\top P^{-1/2})$, $\lambda_{\min}(P^{-1/2}H^\top HP^{-1/2})\!\!\leq\!\!|\phi_{i}|^2\!\!\leq\!\!\lambda_{\max}(P^{-1/2}H^\top H P^{-1/2})$. Then, \eqref{eq:Metz} holds if 
 $\|\hat q_i\!-\!\hat q_0\|_{L^\infty}\!\!\leq\!\!\tfrac{\lambda\vartheta\lambda_{\min}(P^{-1/2}H^\top HP^{-1/2})}{\lambda^{1/2}_{\max}(P^{-1/2}\eta_1\eta_1^\top P^{-1/2})}$,
as \eqref{eq:upperq} claims. Thus the $\phi$-system \eqref{eq:homo_phi_iss_mu-1} is positive \cite{farina2011positive}, and claim c) is proved. 
\end{proof}

\section{Simulation Results}\label{SR}

The simulation is given under a scenario as Fig. \ref{fig:scenario}, in which a Multi-Robot System (MRS), comprised of one leader (robot 0) and three followers (robot $\overline{1,3}$).
The robots are omnidirectional and modeled on the point of mass subjecting to the following dynamics, with their motion in the X- and Y-coordinates considered separately,
\begin{equation}\label{eq:MRS}\begin{aligned}
\begin{array}{ll}
     \dot X_i(t)\!\!=\!\!AX_i(t)\!\!+\!\!Bu_{X_i}(t)\!\!+\!\!Bq_{X_i}(t),  \\
     \dot Y_i(t)\!\!=\!\!AY_i(t)\!\!+\!\!Bu_{Y_i}(t)\!\!+\!\!Bq_{Y_i}(t),
\end{array}\;\;
    A\!\!=\!\!\left(\begin{smallmatrix}
    0 & 1\\
    0 & 0
\end{smallmatrix}\right),\;\; B\!\!=\!\!\left(\begin{smallmatrix}
    0 \\
    1
\end{smallmatrix}\right),\;\; i\!\!=\!\!\overline{0,3}.
\end{aligned}\end{equation}
where $X_i\!\!=\!\!(X_{i,1},X_{i,2})^\top\!\!\!\!\in\!\!\R^2$, $Y_i\!\!=\!\!(Y_{i,1},Y_{i,2})^\top\!\!\!\!\in\!\!\R^2$, $(X_{i,1},Y_{i,1})$ represents the position of robot $i$; $X_{i,2}$ and $Y_{i,2}$ represent the velocity of robot $i$ along X- and Y-coordinate, respectively; $u_{X_i}\!\!\in\!\!\R$ and $u_{Y_i}\!\!\in\!\!\R$ are the control input of robot $i$ with $u_{X_0}\!\!=\!\!u_{Y_0}\!\!=\!\!0$; $q_{X_i},q_{Y_i}\!\!\in\!\!\R$ represents some matched disturbance.


\subsection{Finite-time Non-overshooting Consensus}\label{Sim:non-dis}

In MRS \eqref{eq:MRS}, we let $q_{X_i}\!\!=\!\!q_{Y_i}\!\!=\!\!0$. The leader moves on the smooth barrier surface with a constant velocity, and the followers aim to achieve consensus with the leader. During this process, the followers must be guaranteed a safe move, \textit{i.e.,} followers must keep themself behind the leader in X-coordinate. To this end,
  we aim to design some proper $u_{X_i}$ and $u_{Y_i}$, $i\!\!=\!\!\overline{1,3}$ such that the finite-time leader-following consensus is achieved without overshoots in X-coordinate, \textit{i.e.,} let  $e_X\!\!=\!\!(e_{X_1}^\top,e_{X_2}^\top,e_{X_3}^\top)^\top$, $e_{X_i}\!\!=\!\!X_i\!\!-\!\!X_0$, $e_Y\!\!=\!\!(e_{Y_1}^\top,e_{Y_2}^\top,e_{Y_3}^\top)^\top$, $e_{Y_i}\!\!=\!\!Y_i\!\!-\!\!Y_0$,
  $e_{X}$ and $e_{Y}$  be globally finite-time stable, and for $e_{X}(0)\!\!\in\!\!\Omega\!\!\subset\!\!\Sigma\!\!=\!\!\{e_X\!\!\in\!\!\R^{6}\!:\!(I_3\!\otimes\!\eta_1^\top)e_X\!\!\leq\!\!{0}\}$, $e_{X}(t)\!\!\in\!\!\Omega$, $\forall t\!\!\in\!\!\R_+$, $\eta_1\!\!=\!\!(1,0)^\top$.


\begin{figure}[htbp]
    \centering
    \begin{subfigure}{.5\linewidth}
        \centering
        \includegraphics[width=\linewidth]{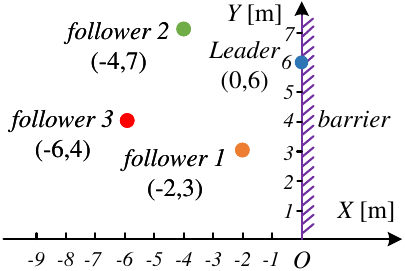}
        \caption{}
    \end{subfigure}
    \begin{subfigure}{.48\linewidth}
        \centering
        \includegraphics[width=.7\linewidth]{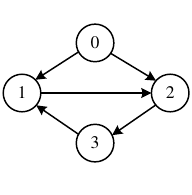}
        \caption{}
    \end{subfigure}
    \caption{The diagram of (a) the MRS with specified initial position $(X_{i,1}(0),Y_{i,1}(0))$, $i\!\!=\!\!\overline{0,3}$; and (b) communication topology.}
    \label{fig:scenario}
\end{figure}

Let $w_{ij}\!\!=\!\!1$ for $(i,j)\!\!\in\!\!\mathcal{E}$.
Let $v_i$ be defined in \eqref{eq:def_vi} with $x_i\!\!=\!\!X_i$, $i\!\!=\!\!\overline{0,3}$ and $x_i\!\!=\!\!Y_i$, $i\!\!=\!\!\overline{0,3}$, using Theorem \ref{thm:state_con}, one has the solution $v_X\!\!=\!\!e_X$, $v_X\!\!=\!\!(v_{X_1}^\top,v_{X_2}^\top,v_{X_3}^\top)^\top\!\!\!\in\!\!\R^{6}$ and $v_Y\!\!=\!\!e_Y$, $v_Y\!\!=\!\!(v_{Y_1}^\top,v_{Y_2}^\top,v_{Y_3}^\top)^\top\!\!\!\in\!\!\R^{6}$, respectively. Let $u_{X_i}$ take form of the finite-time non-overshooting consensus control \eqref{eq:homo_control_nonover}, then
\begin{equation}\label{eq:nonover_MRS}
    u_{X_i}\!\!=\!\!-\|v_{X_i}\|_{\dn}^{1+\mu}K_{lin}\dn(-\ln\|v_{X_i}\|_{\dn})v_{X_i},\;\;v_{X_i}\!\!=\!\!e_{X_i},\;\;i\!\!=\!\!\overline{1,3}.
\end{equation}
Let $u_{Y_i}$ take form of the finite-time consensus control \eqref{eq:homo_control}, then
\begin{equation}\label{eq:consensus_MRS}
    u_{Y_i}\!\!=\!\!-\|v_{Y_i}\|_{\dn}^{1+\mu}K\dn(-\ln\|v_{Y_i}\|_{\dn})v_{Y_i},\;\;v_{Y_i}\!\!=\!\!e_{Y_i},\;\;i\!\!=\!\!\overline{1,3}.
\end{equation}

For the controller \eqref{eq:nonover_MRS}, let $\lambda\!\!=\!\!1$, according to equation \eqref{eq:lin_control}, one has $K_{lin}\!\!=\!\!(\lambda^2,2\lambda)\!\!=\!\!(1,2)$.
Let $\mu\!\!=\!\!-0.2$, then
$G_{\dn}\!\!=\!\!\diag\{1.2,1\}$. YALMIP gives the solution of \eqref{eq:LMI_P} as follows,
\begin{equation*}
    P\!\!=\!\!\left(\begin{smallmatrix}
        0.0020 & 0.0005\\
        0.0005 & 0.0012\\
    \end{smallmatrix}\right).
\end{equation*}
Based on the above parameters, the real-time value of $\|e_{X_i}\|_\dn$ in \eqref{eq:nonover_MRS} is provided by a online solution of
$\|\dn(-\ln\|e_{X_i}\|_\dn)e_{X_i}\|_P\!\!=\!\!1$
using the ``HCS Toolbox for MATLAB'' available in \cite{polyakov_hcs_toolbox}.

For the controller \eqref{eq:consensus_MRS}, let $\mu\!\!=\!\!-0.2$, thus
$G_{\dn}\!\!=\!\!\diag\{1.2,1\}$. YALMIP gives the solution of \eqref{eq:LMI_P_U} as follows,
\begin{equation*}
    X\!\!=\!\!\left(\begin{smallmatrix}
        0.8281 & -0.3107\\
        -0.3107 & 0.9377
    \end{smallmatrix}\right),\quad Y\!\!=\!\!\left(\begin{smallmatrix}
        0.7502 & 0.5000\\
    \end{smallmatrix}\right),
\end{equation*}
immediately we have
\begin{equation*}
P\!\!=\!\!X^{-1}\!\!\!=\!\!\left(\begin{smallmatrix}
        1.3791 & 0.4569\\
        0.4569 & 1.2178
    \end{smallmatrix}\right),\quad
    K\!\!=\!\!YX^{-1}\!\!\!=\!\!\left(\begin{smallmatrix}
1.2630 & 0.9517
\end{smallmatrix}\right).
\end{equation*}
 Similarly, the real-time value of $\|e_{Y_i}\|_\dn$ in \eqref{eq:consensus_MRS} is provided by the online solution of
  $\|\dn(-\ln\|e_{Y_i}\|_\dn)e_{Y_i}\|_P\!\!=\!\!1$.

Let the initial position of the MRS be Fig. \ref{fig:scenario} gives, and the initial velocity $(X_{0,2}(0),Y_{0,2}(0))\!\!=\!\!(0[{\mathrm{m}}/{\mathrm{s}}],1[{\mathrm{m}}/{\mathrm{s}}])$, $(X_{i,2}(0),Y_{i,2}(0))\!\!=\!\!(1[{\mathrm{m}}/{\mathrm{s}}],1[{\mathrm{m}}/{\mathrm{s}}])$, $i\!\!=\!\!\overline{1,3}$. The implicit Euler method is employed to solve the closed-loop dynamics equation \eqref{eq:MRS}, \eqref{eq:nonover_MRS}, \eqref{eq:consensus_MRS} on MATLAB. The performance of non-overshooting finite-time consensus is shown in Fig. \ref{fig:homo_consensus}.
\begin{figure}[htbp]
    \centering
    \includegraphics[width=0.9\linewidth]{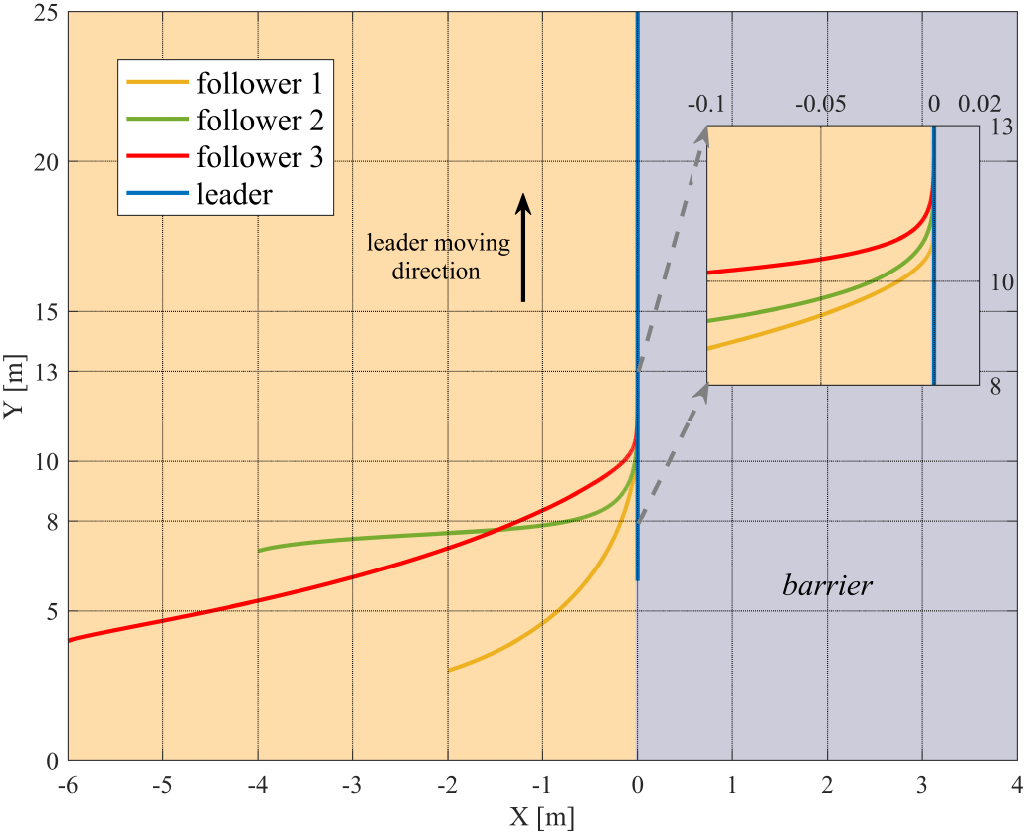}
    \caption{The finite-time non-overshooting leader-following consensus with the homogeneous controller \eqref{eq:nonover_MRS}, \eqref{eq:consensus_MRS}.}
    \label{fig:homo_consensus}
\end{figure}

\subsection{Robust Non-overshooting Finite-time Consensus}

Let the matched disturbances of MRS \eqref{eq:MRS} be defined as below:
$\hat q_{X_0}\!\!\!=\!\!0$, $\hat q_{X_i}\!\!\in\!\!\R$ is uniformly
distributed in the interval $(-0.540, 0.540)$, $(-0.444, 0.444)$ and $(-0.462, 0.462)$ for $i\!\!=\!\!1$, $i\!\!=\!\!2$ and $i\!\!=\!\!3$, respectively;
and $\hat q_{Y_i}\!\!\in\!\!\R$ is
distributed in the interval $(-0.030,0.030)$, $(-0.428, 0.428)$, $(-0.533, 0.533)$ and $(-0.441, 0.441)$ for $i\!\!=\!\!0$, $i\!\!=\!\!1$, $i\!\!=\!\!2$ and $i\!\!=\!\!3$, respectively.

Let $\mu\!\!=\!\!-1$. Let the initial conditions be the same with Section \ref{Sim:non-dis}. The homogeneous controller \eqref{eq:nonover_MRS} and \eqref{eq:consensus_MRS} regulate the followers to reach the leader in a finite time without overshoots. As a comparison,  the linear controller (by letting $\mu\!\!=\!\!0$ for \eqref{eq:nonover_MRS} and \eqref{eq:consensus_MRS}) cannot ensure a safe movement for the MAS. This comparative performance is shown in Fig. \ref{fig:ISS}.


\begin{figure}[htbp]
    \centering
    \begin{subfigure}{\linewidth}
        \centering
        \includegraphics[width=0.9\linewidth]{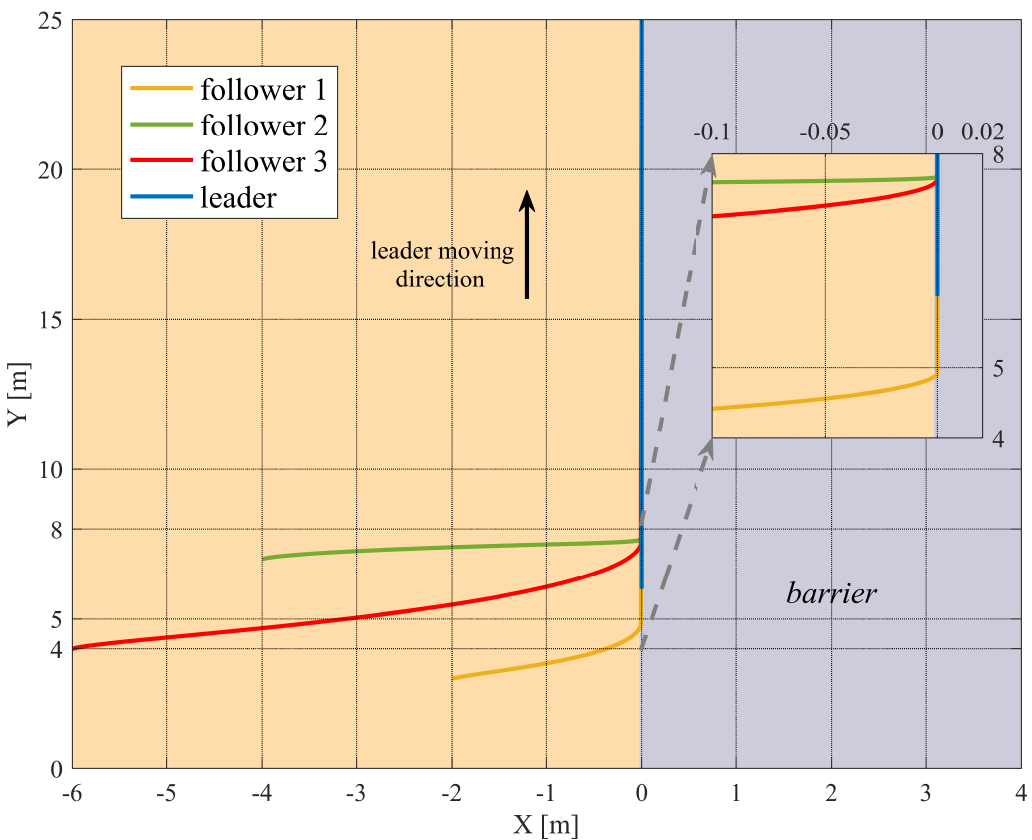}
        \caption{}
    \end{subfigure}
    \hfill
    \begin{subfigure}{\linewidth}
        \centering
        \includegraphics[width=0.9\linewidth]{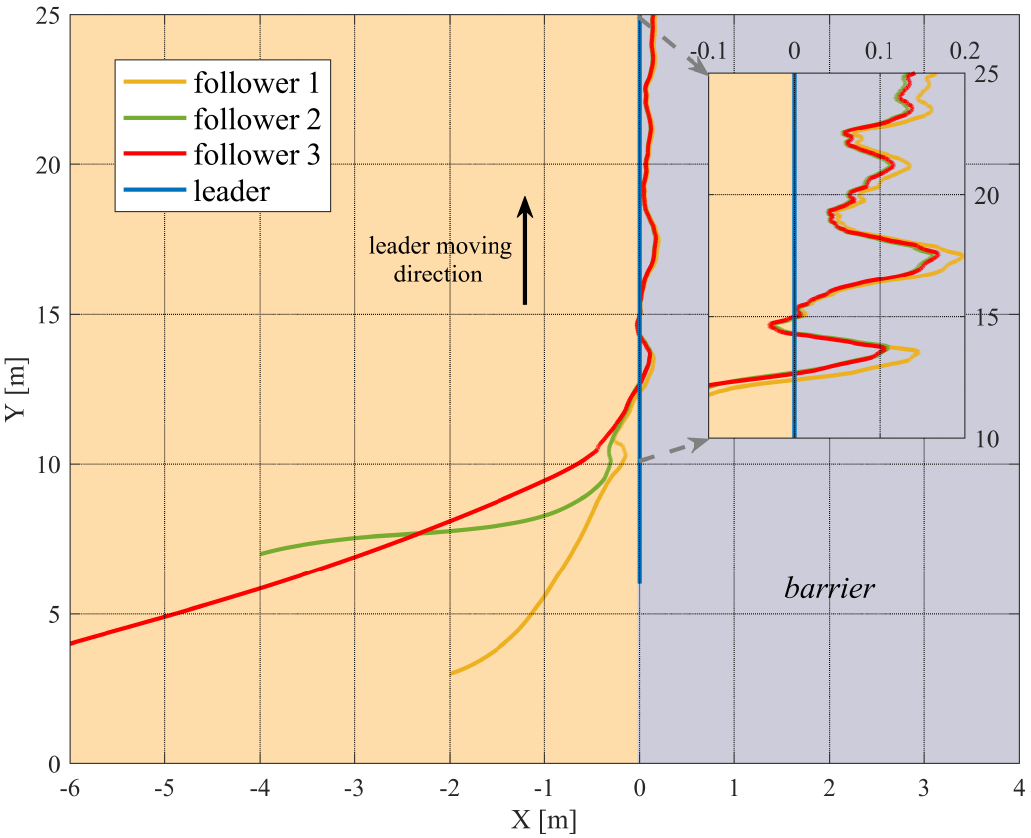}
        \caption{}
    \end{subfigure}
    \caption{The trajectory of the MRS considering disturbance by employing the (a) homogeneous control protocol with $\mu\!\!=\!\!-1$; and (b) linear control protocol.}
    \label{fig:ISS}
\end{figure}

\section{Conclusion}\label{Con}
This paper addresses the finite-time non-overshooting leader-following consensus for MAS. 
A control protocol ensuring an asymptotic non-overshooting consensus is designed. By upgrading the linear protocol into a homogeneous one, the finite-time non-overshooting consensus is achieved. It is further shown that the homogeneous control protocol preserves the non-overshooting performance of the MAS even in the presence of some class of disturbances. Simulation results validate the obtained controls.


\bibliographystyle{IEEEtran}
\bibliography{ref/ref}

\begin{thebibliography}{10}
\providecommand{\url}[1]{#1}
\csname url@samestyle\endcsname
\providecommand{\newblock}{\relax}
\providecommand{\bibinfo}[2]{#2}
\providecommand{\BIBentrySTDinterwordspacing}{\spaceskip=0pt\relax}
\providecommand{\BIBentryALTinterwordstretchfactor}{4}
\providecommand{\BIBentryALTinterwordspacing}{\spaceskip=\fontdimen2\font plus
\BIBentryALTinterwordstretchfactor\fontdimen3\font minus
  \fontdimen4\font\relax}
\providecommand{\BIBforeignlanguage}[2]{{%
\expandafter\ifx\csname l@#1\endcsname\relax
\typeout{** WARNING: IEEEtran.bst: No hyphenation pattern has been}%
\typeout{** loaded for the language `#1'. Using the pattern for}%
\typeout{** the default language instead.}%
\else
\language=\csname l@#1\endcsname
\fi
#2}}
\providecommand{\BIBdecl}{\relax}
\BIBdecl

\bibitem{ni2010leader}
W.~Ni and D.~Cheng, ``Leader-following consensus of multi-agent systems under
  fixed and switching topologies,'' \emph{Systems \& control letters}, vol.~59,
  no. 3-4, pp. 209--217, 2010.

\bibitem{guan2012finite}
Z.-H. Guan, F.-L. Sun, Y.-W. Wang, and T.~Li, ``Finite-time consensus for
  leader-following second-order multi-agent networks,'' \emph{IEEE Transactions
  on Circuits and Systems I: Regular Papers}, vol.~59, no.~11, pp. 2646--2654,
  2012.

\bibitem{li2023generalized}
M.~Li, A.~Polyakov, and G.~Zheng, ``On generalized homogeneous leader-following
  consensus control for multi-agent systems,'' \emph{IEEE Transactions on
  Control of Network Systems}, 2023.

\bibitem{fuller1960relay}
A.~T. Fuller, ``Relay control systems optimized for various performance
  criteria,'' \emph{IFAC Proceedings Volumes}, vol.~1, no.~1, pp. 520--529,
  1960.

\bibitem{korobov1979solution}
V.~I. Korobov, ``A solution of the problem of synthesis using a controllability
  function,'' in \emph{Doklady Akademii Nauk}, vol. 248, no.~5.\hskip 1em plus
  0.5em minus 0.4em\relax Russian Academy of Sciences, 1979, pp. 1051--1055.

\bibitem{haimo1986finite}
V.~T. Haimo, ``Finite time controllers,'' \emph{SIAM Journal on Control and
  Optimization}, vol.~24, no.~4, pp. 760--770, 1986.

\bibitem{bhat2000finite}
S.~P. Bhat and D.~S. Bernstein, ``Finite-time stability of continuous
  autonomous systems,'' \emph{SIAM Journal on Control and optimization},
  vol.~38, no.~3, pp. 751--766, 2000.

\bibitem{wang2008finite}
X.~Wang and Y.~Hong, ``Finite-time consensus for multi-agent networks with
  second-order agent dynamics,'' \emph{IFAC Proceedings volumes}, vol.~41,
  no.~2, pp. 15\,185--15\,190, 2008.

\bibitem{5404774}
A.~Nedic, A.~Ozdaglar, and P.~A. Parrilo, ``Constrained consensus and
  optimization in multi-agent networks,'' \emph{IEEE Transactions on Automatic
  Control}, vol.~55, no.~4, pp. 922--938, 2010.

\bibitem{lin2013constrained}
P.~Lin and W.~Ren, ``Constrained consensus in unbalanced networks with
  communication delays,'' \emph{IEEE Transactions on Automatic Control},
  vol.~59, no.~3, pp. 775--781, 2013.

\bibitem{liu2017constrained}
Q.~Liu, S.~Yang, and Y.~Hong, ``Constrained consensus algorithms with fixed
  step size for distributed convex optimization over multiagent networks,''
  \emph{IEEE Transactions on Automatic Control}, vol.~62, no.~8, pp.
  4259--4265, 2017.

\bibitem{qiu2016distributed}
Z.~Qiu, S.~Liu, and L.~Xie, ``Distributed constrained optimal consensus of
  multi-agent systems,'' \emph{Automatica}, vol.~68, pp. 209--215, 2016.

\bibitem{de2001stabilization}
P.~De~Leenheer and D.~Aeyels, ``Stabilization of positive linear systems,''
  \emph{Systems \& control letters}, vol.~44, no.~4, pp. 259--271, 2001.

\bibitem{rami2007controller}
M.~A. Rami and F.~Tadeo, ``Controller synthesis for positive linear systems
  with bounded controls,'' \emph{IEEE Transactions on Circuits and Systems II:
  Express Briefs}, vol.~54, no.~2, pp. 151--155, 2007.

\bibitem{farina2011positive}
L.~Farina and S.~Rinaldi, \emph{Positive linear systems: theory and
  applications}.\hskip 1em plus 0.5em minus 0.4em\relax John Wiley \& Sons,
  2011.

\bibitem{6681911}
M.~E. Valcher and P.~Misra, ``On the stabilizability and consensus of positive
  homogeneous multi-agent dynamical systems,'' \emph{IEEE Transactions on
  Automatic Control}, vol.~59, no.~7, pp. 1936--1941, 2014.

\bibitem{7892854}
M.~E. Valcher and I.~Zorzan, ``On the consensus of homogeneous multiagent
  systems with positivity constraints,'' \emph{IEEE Transactions on Automatic
  Control}, vol.~62, no.~10, pp. 5096--5110, 2017.

\bibitem{8862888}
J.~J.~R. Liu, J.~Lam, and Z.~Shu, ``Positivity-preserving consensus of
  homogeneous multiagent systems,'' \emph{IEEE Transactions on Automatic
  Control}, vol.~65, no.~6, pp. 2724--2729, 2020.

\bibitem{su2017positive}
H.~Su, H.~Wu, X.~Chen, and M.~Z.~Q. Chen, ``Positive edge consensus of complex
  networks,'' \emph{IEEE Transactions on Systems, Man, and Cybernetics:
  Systems}, vol.~48, no.~12, pp. 2242--2250, 2017.

\bibitem{su2018positive}
H.~Su, H.~Wu, and J.~Lam, ``Positive edge-consensus for nodal networks via
  output feedback,'' \emph{IEEE Transactions on Automatic Control}, vol.~64,
  no.~3, pp. 1244--1249, 2018.

\bibitem{phillips1988conditions}
S.~F. Phillips and D.~E. Seborg, ``Conditions that guarantee no overshoot for
  linear systems,'' \emph{International Journal of Control}, vol.~47, no.~4,
  pp. 1043--1059, 1988.

\bibitem{krstic2006nonovershooting}
M.~Krstic and M.~Bement, ``Nonovershooting control of strict-feedback nonlinear
  systems,'' \emph{IEEE Transactions on Automatic Control}, vol.~51, no.~12,
  pp. 1938--1943, 2006.

\bibitem{schmid2010unified}
R.~Schmid and L.~Ntogramatzidis, ``A unified method for the design of
  nonovershooting linear multivariable state-feedback tracking controllers,''
  \emph{Automatica}, vol.~46, no.~2, pp. 312--321, 2010.

\bibitem{8378231}
R.~Schmid and H.~D. Aghbolagh, ``Nonovershooting cooperative output regulation
  of linear multiagent systems by dynamic output feedback,'' \emph{IEEE
  Transactions on Control of Network Systems}, vol.~6, no.~2, pp. 526--536,
  2019.

\bibitem{ning2020bipartite}
B.~Ning, Q.-L. Han, and Z.~Zuo, ``Bipartite consensus tracking for second-order
  multiagent systems: A time-varying function-based preset-time approach,''
  \emph{IEEE Transactions on Automatic Control}, vol.~66, no.~6, pp.
  2739--2745, 2020.

\bibitem{hong2006tracking}
Y.~Hong, J.~Hu, and L.~Gao, ``Tracking control for multi-agent consensus with
  an active leader and variable topology,'' \emph{Automatica}, vol.~42, no.~7,
  pp. 1177--1182, 2006.

\bibitem{ren2007multi}
W.~Ren, ``Multi-vehicle consensus with a time-varying reference state,''
  \emph{Systems \& Control Letters}, vol.~56, no. 7-8, pp. 474--483, 2007.

\bibitem{liu2011synchronization}
S.~Liu, L.~Xie, and F.~L. Lewis, ``Synchronization of multi-agent systems with
  delayed control input information from neighbors,'' \emph{Automatica},
  vol.~47, no.~10, pp. 2152--2164, 2011.

\bibitem{wang2021generalized}
S.~Wang, A.~Polyakov, and G.~Zheng, ``Generalized homogenization of linear
  controllers: Theory and experiment,'' \emph{International Journal of Robust
  and Nonlinear Control}, vol.~31, no.~9, pp. 3455--3479, 2021.

\bibitem{polyakov2020generalized}
A.~Polyakov, \emph{Generalized homogeneity in systems and control}.\hskip 1em
  plus 0.5em minus 0.4em\relax Springer, 2020.

\bibitem{olfati2004consensus}
R.~Olfati-Saber and R.~M. Murray, ``Consensus problems in networks of agents
  with switching topology and time-delays,'' \emph{IEEE Transactions on
  automatic control}, vol.~49, no.~9, pp. 1520--1533, 2004.

\bibitem{zhu2023robust}
Z.~H. Zhu, H.~Wu, Z.~H. Guan, Z.~W. Liu, Y.~Chen, and X.~Zheng, ``Robust
  prescribed-time coordination control of signed networks with arbitrary
  topologies,'' \emph{IEEE Transactions on Network Science and Engineering},
  2023.

\bibitem{9285180}
Y.~Zhai, Z.-W. Liu, Z.-H. Guan, and Z.~Gao, ``Resilient delayed impulsive
  control for consensus of multiagent networks subject to malicious agents,''
  \emph{IEEE Transactions on Cybernetics}, vol.~52, no.~7, pp. 7196--7205,
  2022.

\bibitem{wang2010finite}
L.~Wang and F.~Xiao, ``Finite-time consensus problems for networks of dynamic
  agents,'' \emph{IEEE Transactions on Automatic Control}, vol.~55, no.~4, pp.
  950--955, 2010.

\bibitem{ames2016control}
A.~D. Ames, X.~Xu, J.~W. Grizzle, and P.~Tabuada, ``Control barrier function
  based quadratic programs for safety critical systems,'' \emph{IEEE
  Transactions on Automatic Control}, vol.~62, no.~8, pp. 3861--3876, 2016.

\bibitem{polyakov2023finite}
A.~Polyakov and M.~Krstic, ``Finite-and fixed-time nonovershooting stabilizers
  and safety filters by homogeneous feedback,'' \emph{IEEE Transactions on
  Automatic Control}, vol.~68, no.~11, pp. 6434--6449, 2023.

\bibitem{orlov2004finite}
Y.~Orlov, ``Finite time stability and robust control synthesis of uncertain
  switched systems,'' \emph{SIAM Journal on Control and Optimization}, vol.~43,
  no.~4, pp. 1253--1271, 2004.

\bibitem{khalil2009lyapunov}
H.~K. Khalil, ``Lyapunov stability,'' \emph{Control systems, robotics and
  automation}, vol.~12, p. 115, 2009.

\bibitem{isidori1985nonlinear}
A.~Isidori, \emph{Nonlinear control systems: an introduction}.\hskip 1em plus
  0.5em minus 0.4em\relax Springer, 1985.

\bibitem{ren2005consensus}
W.~Ren and R.~W. Beard, ``Consensus seeking in multiagent systems under
  dynamically changing interaction topologies,'' \emph{IEEE Transactions on
  automatic control}, vol.~50, no.~5, pp. 655--661, 2005.

\bibitem{guan2012impulsive}
Z.-H. Guan, Z.-W. Liu, G.~Feng, and M.~Jian, ``Impulsive consensus algorithms
  for second-order multi-agent networks with sampled information,''
  \emph{Automatica}, vol.~48, no.~7, pp. 1397--1404, 2012.

\bibitem{Filippov1988:Book}
A.~F. Filippov, \emph{Differential equations with discontinuous righthand
  sides: control systems}.\hskip 1em plus 0.5em minus 0.4em\relax Springer
  Science \& Business Media, 2013, vol.~18.

\bibitem{boyd1994linear}
S.~Boyd, L.~El~Ghaoui, E.~Feron, and V.~Balakrishnan, \emph{Linear matrix
  inequalities in system and control theory}.\hskip 1em plus 0.5em minus
  0.4em\relax SIAM, 1994.

\bibitem{polyakov_hcs_toolbox}
A.~Polyakov, ``{HCS Toolbox for MATLAB},''
  \url{https://gitlab.inria.fr/polyakov/hcs-toolbox-for-matlab}, 2024,
  accessed: 2024-10-10.

\end{thebibliography}

\end{document}